\documentclass{amsart}

\usepackage{helvet}

\usepackage{amsmath,amsxtra,amsthm,amssymb,xr,color}
\usepackage[all]{xy}

\newtheorem{theorem}{Theorem}[section]
\newtheorem{lemma}[theorem]{Lemma}
\newtheorem{conj}[theorem]{Conjecture}
\newtheorem{proposition}[theorem]{Proposition}
\newtheorem{corollary}[theorem]{Corollary}
\newtheorem{defn}[theorem]{Definition}

\newtheorem{remark}[theorem]{Remark}

\newcommand{\ooo}{\frak{O}}

\newcommand{\LLp}{\Lambda^{(p)}}

\newcommand{\Gal}{\operatorname{Gal}}
\newcommand{\Fil}{\operatorname{Fil}}
\newcommand{\DD}{\mathbb{D}}
\newcommand{\Dcris}{\mathbb{D}_{\rm cris}}
\newcommand{\Acris}{\mathbb{A}_{\rm cris}}
\newcommand{\BB}{\mathbb{B}}
\newcommand{\NN}{\mathbb{N}}
\newcommand{\QQ}{\mathbb{Q}}
\newcommand{\Qp}{\mathbb{Q}_p}
\newcommand{\Zp}{\mathbb{Z}_p}
\newcommand{\ZZ}{\mathbb{Z}}
\newcommand{\cA}{\mathcal{A}}
\newcommand{\cB}{\mathcal{B}}

\renewcommand{\AA}{\mathbb{A}}

\newcommand{\FFF}{\mathcal{F}}

\newcommand{\uI}{\underline{I}}
\newcommand{\uJ}{\underline{J}}

\newcommand{\fp}{\mathfrak{p}}

\newcommand{\vp}{\varphi}
\newcommand{\calL}{\mathcal{L}}
\newcommand{\calH}{\mathcal{H}}
\newcommand{\calO}{\mathcal{O}}
\newcommand{\Iw}{\mathrm{Iw}}
\newcommand{\HIw}{H^1_{\mathrm{Iw}}}

\newcommand{\GL}{\mathrm{GL}}

\newcommand{\Brig}{\BB_{{\rm rig},K}^+}
\newcommand{\AQp}{\AA_{K}^+}
\newcommand{\col}{\mathrm{Col}}
\newcommand{\image}{\mathrm{Im}}
\newcommand{\pr}{\mathrm{pr}}
\newcommand{\fM}{\mathfrak{M}}
\newcommand{\Iarith}{\mathbb{I}_{\mathrm{arith}}}
\newcommand{\loc}{\mathrm{loc}}
\newcommand{\cc}{\mathfrak{c}}
\newcommand{\ff}{\mathfrak{f}}
\newcommand{\fI}{\mathfrak{I}}
\newcommand{\Qpn}{\QQ_{p,n}}

\newcommand{\Sel}{\mathrm{Sel}}
\newcommand{\Char}{\mathrm{char}}
\newcommand{\Ind}{\mathrm{Ind}}

\begin{document}

\title[Integral Iwasawa theory]{Integral Iwasawa theory of Galois representations for non-ordinary primes}

\begin{abstract}
In this paper, we study the Iwasawa theory of a motive whose Hodge-Tate weights are $0$ or $1$ (thence in practice, of a motive associated to an abelian variety) at a non-ordinary prime, over the cyclotomic tower of a number field that is either totally real or CM. In particular, under certain technical assumptions, we construct Sprung-type Coleman maps on the local Iwasawa cohomology groups and use them to define (one unconditional and other conjectural) integral $p$-adic $L$-functions and cotorsion Selmer groups. This allows us to reformulate Perrin-Riou's main conjecture in terms of these objects, in the same fashion as Kobayashi's $\pm$-Iwasawa theory for supersingular elliptic curves. By the aid of the theory of \emph{Coleman-adapted Kolyvagin systems} we develop here, we deduce parts of Perrin-Riou's main conjecture from an explicit reciprocity conjecture.
\end{abstract}

\author{K\^az\i m B\"uy\"ukboduk}
\address{K\^az\i m B\"uy\"ukboduk\newline
Ko\c{c} University, Mathematics  \\
Rumeli Feneri Yolu, 34450 Sariyer \\ 
Istanbul, Turkey}
\email{kbuyukboduk@ku.edu.tr}

\author{Antonio Lei}
\address{Antonio Lei\newline
D\'epartement de Math\'ematiques et de Statistique\\
Universit\'e Laval, Pavillion Alexandre-Vachon\\
1045 Avenue de la M\'edecine\\
Qu\'ebec, QC\\
Canada G1V 0A6}
\email{antonio.lei@mat.ulaval.ca}
\thanks{The first author is partially supported by the Turkish Academy of Sciences and T\"UB\.ITAK}

\maketitle

\section{Introduction}

Fix forever an odd rational prime $p$. Let $F$ be either a totally real or a CM number field which is unramified at all primes above $p$. Let $\mathcal{M}_{/F}$ be a motive defined over $F$ which has coefficients in $\QQ$ and whose Hodge-Tate weights are $0$ or $1$. The goal of this article is to study the cyclotomic Iwasawa theory of $\mathcal{M}$ for primes $p$ such that the $p$-adic realization of $\mathcal{M}$ is crystalline but non-ordinary, much in the spirit of the integral theory initiated by Pollack~\cite{pollack03} and Kobayashi~\cite{kobayashi03}. 

The archetypical example of a motive that fits in our treatment is the motive associated to an abelian variety $A$ defined over $F$ which has supersingular reduction at all primes above $p$. In the case when $F=\QQ$ and the variety $A$ is one-dimensional (i.e., an elliptic curve) the plus/minus theory of Kobayashi and Pollack provides us with a satisfactory set of results. Our initial objective writing this article and its companion~\cite{buyukboduklei1} was to extend their work to the general study of supersingular abelian varieties. 

We first follow the ideas due to Sprung \cite{sprung09} to construct \emph{signed} Coleman maps (in \S\ref{sec:coleman} below) for a class of $p$-adic Galois representations that verify certain conditions. We incorporate this construction with Perrin-Riou's (conjectural) treatment of $p$-adic $L$-functions so as to
\begin{itemize}
\item provide a definition of the signed (integral) $p$-adic $L$-functions attached to motives at non-ordinary primes (see particularly Definition \ref{def:signedpadicLfunc} and Theorem~\ref{thm:decomposepadicL}), conditional on the \emph{Explicit Reciprocity Conjecture}~\ref{conj:Tspecialelement} for the Kolyvagin determinants (as defined in Appendix~\ref{sec:appendixKSexists}), 
\item formulate a signed main conjecture in this setting (Conjecture~\ref{myTCP}) that is equivalent to Perrin-Riou's main conjecture~\cite[\S4]{perrinriou95};
\item utilizing the theory of Coleman-adapted Kolyvagin systems that we develop in Appendix~\ref{sec:appendixKSexists} and assuming the Explicit Reciprocity Conjecture~\ref{conj:Tspecialelement}, verify one containment of the signed main conjecture  (see Theorem~\ref{thm:PRmainconj}) and  deduce a similar result on Perrin-Riou's main conjecture.
\end{itemize}
Note that although we work and state our results in the realm of motives, one of our hypothesis (denoted by (H.F.-L.) below) would essentially force us to restrict our attention to abelian varieties. 

We shall explain our results in detail below. Let us first introduce some notation. 

\subsection{Setup and notation}
For any field $k$, let $\overline{k}$ denote a fixed separable closure of $k$ and $G_k:=\Gal(\overline{k}/k)$ denote its absolute Galois group. Fix forever a $G_F$-stable $\ZZ_p$-lattice $T$ contained inside $\mathcal{M}_p$,  the $p$-adic realization of $\mathcal{M}$.  Let $\mathcal{M}^*(1)$ denote the dual motive and write $T^\dagger=\textup{Hom}(T,\ZZ_p(1))$ for the Cartier dual of $T$.

Let $g:=\dim_{\QQ_p}\left(\textup{Ind}_{F/\QQ} \mathcal{M}_p\right)$ and let $g_+:=\dim_{\QQ_p}\left(\textup{Ind}_{F/\QQ}\, \mathcal{M}_p\right)^{+}$, the dimension of the $+1$-eigenspace under the action of a fixed complex conjugation on $\textup{Ind}_{F/\QQ}\, \mathcal{M}_p$. Set $g_-=g-g_+$. Similarly for any prime $\frak{p}$ of $F$ above $p$, define $g_\frak{p}:=\dim_{\QQ_p}\left(\textup{Ind}_{F_\frak{p}/\QQ_p} \mathcal{M}_p\right)$ so that $g=\sum_{\frak{p}|p}\,g_\frak{p}$.  

For any unramified extension $K$ of $\QQ_p$ that contains $F$, we write $\DD_K(T)$ for its Dieudonn\'e module over $K$, namely $(\Acris\otimes T)^{G_K}$, where $\Acris$ is one of Fontaine's ring. We shall fix a $\ZZ_p$-basis $\frak{B}=\{v_i\}$ of this module. 

\subsubsection{Iwasawa algebras}

Let $\Gamma$ be the Galois group $\Gal(\Qp(\mu_{p^\infty})/\Qp)$. Given any unramified extension $K$ of $\Qp$, we shall abuse notation and write $\Gamma$ for the Galois group $\Gal(K(\mu_{p^\infty})/K)$ as well. We may decompose $\Gamma$ as $\Delta\times\overline{\langle\gamma\rangle}$, where $\Delta$ is cyclic of order $p-1$ and $\overline{\langle\gamma\rangle}$ is isomorphic to the additive group $\Zp$. We write $\Lambda$ for the Iwasawa algebra $\Zp[[\Gamma]]$. We may identify it with the set of power series $\sum_{n\ge0,\sigma\in\Delta}a_{n,\sigma}\cdot\sigma\cdot(\gamma-1)^n$ where $a_{n,\sigma}\in\Zp$. We shall identify $\gamma-1$ with the indeterminate $X$.

 For $n\ge0$, we write $\Qpn=\Qp(\mu_{p^n})$ and $G_n=\Gal(\Qpn/\Qp)$. Denote $ \Zp[G_n]$ by $\Lambda_n$. We have in particular $\Lambda=\varprojlim\Lambda_n$. For any field $k$, define $\HIw(k,T)$ to be $\varprojlim H^1(k(\mu_{p^n}),T)$, where the limit is taken with respect to the corestriction maps.

We define $\calH$ to be the set of elements $\sum_{n\ge0,\sigma\in\Delta}a_{n,\sigma}\cdot\sigma\cdot(\gamma-1)^n$ where $a_{n,\sigma}\in\Qp$ are such that the power series $\sum_{n\ge0}a_{n,\sigma}X^n$ converges on the open unit disc for all $\sigma\in\Delta$.

Let $|\bullet|_p$ denote the normalised $p$-adic norm with $|p|_p=1/p$. For a real number $h\ge0$ and an element $F=\sum_{n\ge0,\sigma\in\Delta}a_{n,\sigma}\cdot\sigma\cdot(\gamma-1)^n\in\calH$, if $\sup_{n\ge1}\frac{|a_{n,\sigma}|_p}{n^h}<\infty$ for all $\sigma\in\Delta$, we say that $F$ is $O(\log^h)$. 


\subsubsection{Isotypic components and characteristic ideals}

Let $M$ be a $\Lambda$-module, $\eta$ a Dirichlet character modulo $p$. We write $e_\eta=\frac{1}{p-1}\sum_{\sigma\in \Delta}\eta(\sigma)^{-1}\sigma\in\Zp[\Delta]$. The $\eta$-isotypic component of $M$ is defined to be $e_\eta\cdot M$ and denoted by $M^\eta$. Note that we may regard $M^\eta$ as a $\Zp[[X]]$-module.

Following \cite{perrinriou95}, we write $e_+$ and $e_-$ for the idempotents $(1+c)/1$ and $(1-c)/2$ respectively, where $c$ is the complex conjugation of $\Delta$. For any $\Lambda$-module $M$, we write $M_\pm=e_\pm M$.

Given an element $F=\sum_{n\ge0,\sigma\in\Delta}a_{n,\sigma}\cdot\sigma\cdot(\gamma-1)^n$ of $\calH$, we shall identify $e_\eta\cdot F$ with the element
\[
\sum_{n\ge0}\left(\sum_{\sigma\in\Delta}a_{n,\sigma}\eta(\sigma)\right)X^n
\in\Qp[[X]].
\]

Given a finitely generated torsion $\Zp[[X]]$-module $N$, we write $\Char_{\Zp[[X]]}N$ for its characteristic ideal. 


\subsection{Statements of the results}
\begin{theorem}[Corollary~\ref{cor:decompose} and (\ref{eq:decomp}) below]\label{thm:introcolemanmaps} Let $\frak{p}$ be a prime of $F$ above $p$. Fix a $\ZZ_p$-basis $\{v_i\}$ of $\DD_{F_\frak{p}}(T)$. Assume that the Hodge-Tate weights of $T|F_\fp$ are inside $\{0,1\}$ and that the Frobenius on $\DD_{F_\fp}(T)$ have slope inside $(-1,0]$ and $1$ is not an eigenvalue. There exists a $\Lambda$-module homomorphism 
$$\textup{Col}_{T|F_\frak{p}}:\HIw(F_\frak{p},T)\longrightarrow \Lambda^{\oplus g_\frak{p}}$$ 
and a matrix $M_{T|F_\frak{p}} \in M_{g_\frak{p}\times g_\fp}(\mathcal{H})$ such that we have the following decomposition of Perrin-Riou's regulator map $\calL_T^{F_\frak{p}}$ (defined as in \S\ref{S:PRexp} below):
$$\calL_T^{F_\frak{p}}=\begin{pmatrix}
v_1&\cdots& v_{g_\frak{p}}
\end{pmatrix}\cdot M_{T|F_\frak{p}}\cdot\col_{T|F_\frak{p}}\,.$$
Here, $\begin{pmatrix}
v_1&\cdots& v_{g_\frak{p}}
\end{pmatrix}$ and $\textup{Col}_{T|F_\frak{p}}$ are regarded as a row vector and a column vector respectively.
\end{theorem}
See \S\ref{subsec:imagesofcol} and Corollary~\ref{cor:intersectionofcolemanmaps} for a very detailed discussion on the  kernels and images of the Coleman maps $\textup{Col}_{T|F_\frak{p}}$. In particular, we are able to prove (see propositions~\ref{prop:iageofcolemanis} and \ref{prop:admissiblebasisexists} below) that the Coleman maps are pseudo-surjective if we choose the basis $\{v_i\}$ suitably.

In addition to the assumptions on $T$ above, assume that the following hypotheses hold true: 
\begin{itemize}
\item[(H.Leop)] $T$ satisfies the weak Leopoldt conjecture, as stated in \cite[\S1.3]{perrinriou95}.
\item[(H.nA)] For every prime $\fp$ of $F$ above $p$. we have 
$$H^0(F_\fp,T/pT)=H^2(F_\fp,T/pT)=0\,.$$
\end{itemize}

Let $\DD_{p}(T)$ be the direct sum $\oplus_{\fp|p}\DD_{F_\fp}(T)$. We assume until the end that the following (weak) form of the \emph{Panchishkin condition} holds true:
\begin{itemize}
\item[(H.P.)] $\dim \,\left( \textup{Fil}^0 \DD_p(T)\otimes\QQ_p\right)\,=g_-$\,.
\end{itemize}
\begin{remark}
Note that the hypotheses \textup{(H.nA)} and \textup{(H.P.)} hold true for the $p$-adic Tate-module of an abelian variety defined over $F$. The hypothesis \textup{(H.Leop)} is expected to hold for any $T$.
\end{remark}

{\begin{remark}
\label{rem:folkloreimpliesslefdualandHPisfree}
Suppose $\mathcal{M}$ is irreducible and (pure) of weight $w$. Let $r_i$ denote the total multiplicity of the Hodge-Tate weight $i$ of the representation $\mathcal{M}_p$ for $i=0,1$. Then
\begin{equation}\label{eqn:weights}\displaystyle{2r_1=2\,\sum_i\, ir_i = wg\,.}
\end{equation}
Furthermore, if we further assumed the truth of Tate's conjecture for $\mathcal{M}_p$, it would follow that $r_0=r_1$. This combined with (\ref{eqn:weights}) shows that $w=1$ 
and $r_0=r_1=g/2$; and Faltings' theorem comparing Hodge and Hodge-Tate weights shows that $g_-=g_+=g/2$. In particular, the condition (H.P.) is automatically verified in our setting if we assume the truth of Tate's conjecture.
\end{remark}}
Let $\underline{I} \subset \{1,\cdots, g\}$ be any subset of size $g_-$. Using the Coleman maps $\textup{Col}_{T|F_\frak{p}}$, we may define (see Definiton \ref{def:signedpadicLfunc}) the \emph{multi-signed} (integral) $p$-adic $L$-function 
$$L_{\uI}(\mathcal{M}^*(1)) \in \Lambda\,.$$ 
We do not provide its precise definition here in the introduction but contend ourselves to the remark that its definition relies on the truth of the \emph{explicit reciprocity conjecture for the Kolyvagin determinants}~(Conjecture~\ref{conj:Tspecialelement}), which we implicitly assume henceforth in this introduction.
We may also use the Coleman maps to define the \emph{multi-signed Selmer groups} $\Sel_{\uI}(T^\dagger/F(\mu_{p^\infty}))$ as in Definition~\ref{def:modifiedselmergroup}.

Suppose until the end of this Introduction that the basis of $\DD_p(T)$ we have fixed as in the statement of Theorem~\ref{thm:introcolemanmaps} is strongly admissible in the sense of Definition~\ref{def:admissiblebasis}. We prove in Appendix~\ref{appendix:linearalgebra} that a strongly admissible basis always exists.
\begin{theorem}[Theorem~\ref{thm:TmyCPeqivtoCP} below]
\label{thm:modifiedmainconjequivalanttoPRsmainconj}
For every even Dirichlet character $\eta$ of $\Delta$ and every $\underline{I}$ as above, the following assertion is equivalent to $\eta$-part of Perrin-Riou's Main Conjecture~\ref{TCP}:
\begin{equation}\label{eqn:signedmainconj}
\Char_{\Zp[[X]]}\left(\Sel_{\uI}(T^\dagger/F(\mu_{p^\infty}))^{\vee,\eta}\right)=L_{\uI}(\mathcal{M}^*(1))^\eta\cdot\Zp[[X]]\,.
\end{equation}
\end{theorem}
The assertion~(\ref{eqn:signedmainconj}) in the statement of Theorem~\ref{thm:modifiedmainconjequivalanttoPRsmainconj} will be referred to as the \emph{signed main conjecture}.

In Appendix~\ref{sec:appendixKSexists}, we develop the theory of Coleman-adapted Kolyvagin systems and prove the existence of what we call an $\mathbb{L}$-restricted Kolyvagin system (see Theorem~\ref{thm:appendixBmain}). Using these objects we define a canonical submodule $\frak{K}(T)\subset\HIw(F_p,T)$, the \emph{module of Kolyvagin determinants}\footnote{We expect that this module should be closely related to the higher rank Kolyvagin systems as studied in \cite{mrksrankr}.}. Assuming the Reciprocity Conjecture~\ref{conj:Tspecialelement} on Kolyvagin determinants, we are able to prove the following portion of the signed main conjecture and Perrin-Riou's main conjecture:
\begin{theorem}[See Theorem~\ref{thm:PRmainconj} and its proof below]\label{thm:mainINTRO} Under the hypotheses of Theorem~\ref{thm:modifiedmainconjequivalanttoPRsmainconj} and the hypotheses \textup{\textbf{(H1)-(H4)}} of \cite[\S3.5]{mr02} on $T$, the containment
$$L_{\uI}(\mathcal{M}^*(1))^\eta\cdot\Zp[[X]]\subset \Char_{\Zp[[X]]}\left(\Sel_{\uI}(T^\dagger/F(\mu_{p^\infty}))^{\vee,\eta}\right)
$$ in \textup{(\ref{eqn:signedmainconj})} and the containment 
\begin{equation}\label{PRmainconjcontainment}e_\eta\cdot L_p(\mathcal{M}^*(1))\cdot\Lambda\subset e_\eta\cdot\Iarith(\mathcal{M})\end{equation}
in the statement of Perrin-Riou's Main Conjecture~\ref{TCP} hold true for every even Dirichlet character $\eta$ of $\Delta$.
\end{theorem}
\begin{remark}
\label{rem:hypoholdtrueINTRO}
See \cite{buyukboduklei1} for an example where we obtain an explicit version of Theorem~\ref{thm:mainINTRO}. In loc.cit., we study more closely the motive attached to the Hecke character associated to a CM abelian variety that has supersingular reduction at all primes above $p$. In this particular case, the hypotheses  \textup{\textbf{(H1)-(H4)}} of \cite[\S3.5]{mr02}, \textup{(H.F.-L.)}, \textup{(H.S.)}, \textup{(H.P.)} and \textup{(H.nA)} hold true. The (conjectural) special elements in that setting are expected to be a form of (conjectural) Rubin-Stark elements.
\end{remark}

\begin{remark}
In order to deduce the containment \eqref{PRmainconjcontainment} for odd characters $\eta$ of $\Delta$, one needs to replace $g_-$ with $g_+$ everywhere. Note also that upon studying the motive $\mathcal{M}\otimes\omega$ (where $\omega$ is the Teichm\"uller character) in place of $\mathcal{M}$, one may reduce the consideration for odd characters to the case of even characters.

To deduce the assertion \eqref{PRmainconjcontainment} for every character $\eta$ of $\Delta$ (and therefore, by the semi-simplicity of $\ZZ_p[\Delta]$, to conclude with the containment $\Lambda\cdot L_p(A^\vee)\subset \Iarith(A)$ in Conjecture~\ref{TCP}), we would need in our proof that $g_-=g_+$, as a result of our running hypothesis \textup{(H.P.)}\,. Note that this condition holds true for motives associated to abelian varieties. 
\end{remark}

\section*{Acknowledgement}
The authors would like to thank Robert Pollack and Florian Sprung for answering their questions during the preparation of this work, as well as Byoung Du Kim for his comments on an earlier version of this article.


\section{Construction of Coleman maps}


In this section, we generalize the construction of \emph{signed} Coleman maps in \cite{kobayashi03,sprung09} to higher dimensional $p$-adic representations that satisfy certain hypotheses. These maps decompose the regulator map of Perrin-Riou, which we recall below.


\subsection{Perrin-Riou's regulator map}\label{S:PRexp}

Let $T$ be a free $\Zp$-module of rank $d$ that is equipped with a crystalline continuous action by the absolute Galois group of a finite unramified extension $K$ of $\Qp$ whose Hodge-Tate weights are all non-negative. 

Let $r=[K:\Qp]$. Recall that we write $\DD_K(T)$ for its Dieudonn\'e module and $\HIw(K,T):=\varprojlim H^1(K(\mu_{p^n}),T)$.

Let
\[
\langle\sim,\sim\rangle_n:H^1((K(\mu_{p^n}),T)\times H^1((K(\mu_{p^n}),T^*(1))\rightarrow\Zp
\]
be the local Tate pairing for $n\ge0$. This gives a pairing
\begin{align*}
\langle\sim,\sim\rangle:\HIw(K,T)\times \HIw(K,T^*(1))&\rightarrow\Lambda\\
((x_n)_n,(y_n)_n)&\mapsto\left(\sum_{\sigma\in G_n}\langle x_n,y_n^\sigma\rangle_n\cdot\sigma\right)_n,
\end{align*}
which can be extended $\calH$-linearly to a pairing
\[
\langle\sim,\sim\rangle:\calH\otimes_\Lambda\HIw(K,T)\times \calH\otimes_\Lambda\HIw(K,T^*(1))\rightarrow\calH.
\]

Let 
$$\calL_T^K:\HIw(K,T)\rightarrow\calH\otimes_{\Zp}\DD_K(T)$$ 
be Perrin-Riou's $p$-adic regulator given as in \cite[Definition~3.4]{leiloefflerzerbes11}. In the case where the eigenvalues of $\vp$ on $\DD_K(T)$ are not powers of $p$, we may describe this map concretely as follows. Fix a $\Zp$-basis $v_1,\ldots, v_{rd}$ of $\DD_K(T)$ and let $v_1',\ldots,v_{rd}'$ be the dual basis of $\DD_K(T^*(1))$. For $i\in\{1,\ldots,rd\}$, we write $\calL_{T,i}^K:\HIw(T)\rightarrow\calH$ for the map obtained by composing $\calL^K_T$ and the projection of $\calH\otimes\DD_K(T)$ to the $v_i$-component. The Colmez-Perrin-Riou reciprocity law (stated in \cite{perrinriou94} and proved in \cite{colmez98}) implies that
\begin{equation}
\calL^K_{T,i}(z)=\langle z,\Omega_{T^*(1)}( v_i')\rangle,
\label{eq:regulator}
\end{equation}
where $\Omega_{T^*(1)}$ is the Perrin-Riou exponential map
\[
\Omega_{T^*(1)}:\calH\otimes_{\Zp}\DD(T^*(1))\rightarrow\calH\otimes_{\Zp}\HIw(T^*(1))
\]
defined in \cite{perrinriou94}.  Note that our assumption on the eigenvalues of $\vp$ means that we may state the properties of Perrin-Riou's exponential map in a slightly simpler way than \cite{perrinriou94}. Recall that if $\theta$ is a Dirichlet character of conductor $p^n$, \cite[Lemma~3.5]{lei09} implies that
\begin{equation}\label{PRinterpolation}
\theta(\calL^K_{T,i}(z))=
\begin{cases}
\left[\exp^*_0(z),(1-p^{-1}\vp^{-1})(1-\vp)^{-1}v_i'\right]&\text{if $n=0$,}\\
\frac{1}{\tau(\theta^{-1})}\left[\sum_{\sigma\in G_n}\theta^{-1}(\sigma)\exp^*_n(z^\sigma),\vp^{-n}(v_i')\right]&\text{otherwise}
\end{cases}
\end{equation}
where $[\sim,\sim]$ is the natural pairing
\[
\DD_K(T)\times\DD_K(T^*(1))\rightarrow\Zp,
\]
which is extended linearly to
\[
\Qpn\otimes_{\Zp}\DD_K(T)\times\Qpn\otimes_{\Zp}\DD_K(T^*(1))\rightarrow\Qpn.
\]

In order to define the \emph{signed} Coleman maps, we assume further that $T$ verifies the following conditions.

\begin{itemize}
\item[(H.F.-L.)] The Hodge-Tate weights of $T$ are $0$ and $1$.
\item[(H.S.)] The slopes of $\vp$ on $\DD_K(T)$ lie in the interval $(-1,0]$ and $1$ is not an eigenvalue.
\end{itemize}

\begin{remark}
These hypotheses ensure that the eigenvalues of $\vp$ are not integral powers of $p$.
\end{remark}
\begin{remark} Note that both of these hypotheses are satisfied by the $p$-adic Tate module of an abelian variety which has supersingular reduction at all primes above $p$. In fact, note that the hypothesis \textup{(H.F.-L.)} would essentially restrict the extent of our treatment to abelian varieties.  
\end{remark}
\begin{remark} 
The hypothesis (H.F.-L.) implies that $T$ is Fontaine-Laffaille. Hence,
\begin{equation}\label{eq:FL3}
\vp(\DD_K(T))\subset\frac{1}{p}\DD_K(T)\quad\text{and}\quad\vp(\Fil^0\DD_K(T))\subset\DD_K(T)
\end{equation}
Moreover,
\begin{equation}\label{eq:FL4}
\DD_K(T)=p\vp(\DD_K(T))+\vp(\Fil^0\DD_K(T))
\end{equation}
\end{remark}

\subsection{Logarithmic matrix}\label{S:matrix}

We fix a $\Zp$-basis $v_1,v_2,\ldots,v_{rd}$ of $\DD_K(T)$ such that $v_1,\ldots,v_{rd_0}$ is a basis of $\Fil^0\DD_K(T)$. Let $C_\vp$ be the matrix of $\vp$ with respect to this basis. By \eqref{eq:FL3} and \eqref{eq:FL4}, $C_\vp$ is of the form
{
\begin{equation}\label{eq:Aphi}
C\left(
\begin{array}{c|c}
I_{rd_0}&0\\ \hline
0&\frac{1}{p}I_{r(d-d_0)}
\end{array}
\right)
\end{equation}}
for some $C\in \GL_{rd}(\Zp)$. We note in particular that $C_\vp^{-1}$ is defined over $\Zp$. 

For $n\ge1$, we write $\Phi_{p^n}(1+X)$ for the cyclotomic polynomial 
\[
\sum_{i=0}^{p-1}(1+X)^{ip^{n-1}}
\]
and $\omega_n(X)=(1+X)^{p^n}-1$.

\begin{defn}
For $n\ge1$, we define
\[
C_n=
\left(
\begin{array}{c|c}
I_{rd_0}&0\\ \hline
0&\Phi_{p^n}(1+X)I_{r(d-d_0)}
\end{array}
\right)C^{-1}
\quad\text{and}\quad
M_n=\left(C_\vp\right)^{n+1}C_{n}\cdots C_{1}.
\]
\end{defn}

\begin{proposition}\label{prop:matrix}
The sequence of matrices $\{M_n\}_{n\ge1}$ converges entry-wise with respect to the sup-norm topology on $\calH$. If $M_T$ denotes the limit of the sequence, each entry of $M_T$ are $o(\log)$. Moreover, $\det(M_T)$ is, up to a constant in $\Zp^\times$, equal to $\left(\frac{\log(1+X)}{pX}\right)^{r(d-d_0)}$.
\end{proposition}
\begin{proof}
For all $m>n$, we have
\[
\Phi_{p^m}(1+X)\equiv p\mod \omega_n,
\]
which implies that
\[
C_m\equiv \left(C_\vp\right)^{-1}\mod\omega_n.
\]
Therefore, we deduce that
\[
M_{m}\equiv M_n\mod\omega_n.
\]

Note that all entries of $C_1\cdots C_{n}$ are in $\Zp[[X]]$. By (H.S.), there exists a constant $h<1$ such that $v_p(\alpha)\ge -h$ for all eigenvalues of $\alpha$ of $C_\vp$. Therefore, all entries of  $\left(C_\vp\right)^{n+1}$ are in $\frac{R}{p^{ nh}}\Zp$ for some constant $R$. The coefficients of the entries of $M_n$ are $O(p^{-nh})$, so the result follows from \cite[\S1.2.1]{perrinriou94}.
\end{proof}

\begin{remark}
The matrix $M_T$ is uniquely determined by the matrix $C$.
\end{remark}

\begin{lemma}\label{lem:evaluatematrix}
If $\eta$ is a character on $\Delta$, then $\eta(M_T)=C_\vp$.
\end{lemma}
\begin{proof}
Since $\eta(\Phi_{p^n})=p$ for all $n\ge1$, we have $\eta(C_n)=(C_\vp)^{-1}$. This implies $\eta(M_n)=C_\vp$, hence the result. 
\end{proof}


\subsection{Decomposing Perrin-Riou's regulator map}\label{sec:coleman}

We shall use the matrix $M_T$ to decompose Perrin-Riou's regulator map in the following sense. For all $z\in\HIw(K,T)$, we shall find $\col^K_T(z)\in\Lambda^{\oplus r d}$ such that
\[
\calL^K_T(z)=\begin{pmatrix}
v_1&\cdots& v_{rd}
\end{pmatrix}\cdot M_T\cdot\col^K_T(z).
\]
Throughout this section, we shall fix an element $z\in\HIw(K,T)$. Its image under Perrin-Riou's regulator has the following interpolation properties.

\begin{lemma}\label{lem:regulatorformula}
If $\theta$ is a Dirichlet character of conductor $p^n$, then
\[
\theta(\calL^K_T(z))=
\begin{cases}
\sum_{i=1}^{rd}\left[\exp^*_0(z),v_i'\right](1-\vp)(1-p^{-1}\vp^{-1})^{-1}(v_i)&\text{if $n=0$,}\\
\frac{p^n}{\tau(\theta^{-1})}\sum_{i=1}^{rd}\left[\sum_{\sigma\in G_n}\theta^{-1}(\sigma)\exp^*_n(z^\sigma),v_i'\right]\vp^{n}(v_i)&\text{otherwise}.
\end{cases}
\]
\end{lemma}
\begin{proof}
Note that the adjoints of $(1-p^{-1}\vp^{-1})(1-\vp)^{-1}$ and $\vp^{-1}$ under $[\sim,\sim]$ are $(1-\vp)(1-p^{-1}\vp^{-1})^{-1}$ and $p\vp$ respectively. Hence, the result follows from \eqref{PRinterpolation}.
\end{proof}

\begin{proposition}\label{prop:integral}
For $n\ge1$, there exists a unique $\calL^{(n)}_T(z)\in\Lambda_n\otimes_{\Zp}\DD_K(T)$ such that \[\vp^{-n-1}\left(\calL^K_T(z)\right)\equiv\calL_T^{(n)}(z)\mod\omega_n.\]
\end{proposition}
\begin{proof}
Recall from \cite[\S3.1]{leiloefflerzerbes11} that the map $\calL^K_T$ is given by
\[
(\fM^{-1}\otimes 1)\circ(1-\vp)\circ (h^1_T)^{-1},
\]
where $\fM$ is the Mellin transform that sends each element of $\calH$ to some convergent power series in $\pi$ and $h^1_T$ is the isomorphism of Berger \cite[\S A]{berger03} between $H^1_{\Iw}(K,T)$ and $\NN(T)^{\psi=1}$, with $\NN(T)$ being the Wach module of $T$. Under Mellin transform, integrality is preserved and the ideal generated by $\omega_n$ corresponds to the one generated by $\vp^{n+1}(\pi)$ (c.f. \cite[Theorem~5.4]{leiloefflerzerbes10}). Hence, the proposition follows from Lemma~\ref{lem:integralityPR} in the appendix.
\end{proof}

We write $\calL_{T,1}^{(n)}(z),\ldots,\calL_{T,rd}^{(n)}(z)$ for the elements in $\Lambda_n$ that are given by the projections of $\calL^{(n)}_T(z)\mod\omega_n$ to the $v_i$-component as $i$ runs from $1$ to $rd$. From Proposition~\ref{prop:integral}, we have the congruence
\begin{equation}\label{eq:Lcongruence}
(C_\vp)^{-n-1}\cdot\begin{pmatrix}
\calL^K_{T,1}(z)\\
\vdots\\
\calL^K_{T,rd}(z)
\end{pmatrix}\equiv
\begin{pmatrix}
\calL_{T,1}^{(n)}(z)\\
\vdots\\
\calL_{T,rd}^{(n)}(z)
\end{pmatrix}
\mod\omega_n.
\end{equation}

For $n\ge1$, we identify $\Lambda_n^{\oplus rd}$ with the column vectors of dimension $rd$ with entries in $\Lambda_n$. Define $h_n$ to be the $\Lambda_n$-endomorphism on $\Lambda_n^{\oplus rd}$ given by the left multiplication by the product of matrices $C_n\cdots C_1$. Let $\pi_n$ denote the projection map $\Lambda_{n+1}^{\oplus rd}\rightarrow\Lambda_{n}^{\oplus rd}$.  

\begin{proposition}\label{prop:factorisation}
For $n\ge1$, there exists a unique element $\col^{(n)}_T(z)\in\Lambda_n^{\oplus rd}/\ker h_n$ such that
\[
\begin{pmatrix}
\calL_{T,1}^{(n)}(z)\\
\vdots\\
\calL_{T,rd}^{(n)}(z)
\end{pmatrix}
\equiv
C_n\cdots C_1\cdot
\col^{(n)}_T
\mod\ker h_n.
\]
\end{proposition}
\begin{proof}
By \cite[Proposition~4.8]{leiloefflerzerbes11}, if $\theta$ is a Dirichlet character of conductor $p^{n+1}$, then $\theta\left(\vp^{-n-1}\left(\calL^K_T(z)\right)\right)\in\Qpn\otimes_{\Zp}\Fil^0\DD_K(T)$. In other words, $\vp^{-n-1}\left(\calL^K_T(z)\right)$ is of the form $\sum_{i=1}^{rd}F_iv_i$ for some $F_i\in \calH$ where $\Phi_{p^n}(1+X)|F_i$ for $i=rd_0+1,\ldots,rd$. 
But
\[
\vp^{-n-1}\left(\calL^K_T(z)\right)=
\begin{pmatrix}
v_1&\cdots&v_{rd}
\end{pmatrix}\cdot (C_\vp)^{-n-1}\cdot\begin{pmatrix}
\calL^K_{T,1}(z)\\
\vdots\\
\calL^K_{T,rd}(z)
\end{pmatrix}.
\]

Therefore, on combining this with \eqref{eq:Lcongruence}, we deduce that $\calL_{T,rd_0+1}^{(n)}(z),\ldots,\calL_{T,rd}^{(n)}(z)$ are all divisible by $\Phi_{p^n}(1+X)$. Hence, there exists a unique element $\col^{(n,1)}_T(z)\in\Lambda_n^{\oplus rd}/\ker C_n$ such that
\[
\begin{pmatrix}\calL_{T,1}^{(n)}(z)\\ \vdots\\ \calL_{T,rd}^{(n)}(z)\end{pmatrix}
\equiv
C_n\cdot\col^{(n,1)}_T(z)\mod\ker C_n.
\]

But $C_n\equiv \left(C_\vp\right)^{-1}$ (which is defined over $\Zp$) modulo $\omega_{n-1}$, so 
\[
\col^{(n,1)}_T(z)\equiv
\left(C_\vp\right)^{-n}\cdot
\begin{pmatrix}
\calL^K_{T,1}(z)\\
\vdots\\
\calL^K_{T,rd}(z)
\end{pmatrix}\mod (\omega_{n-1},\ker C_n).
\]
Once again, by \cite[Proposition~4.8]{leiloefflerzerbes11}, we may find $\col^{(n,2)}_T(z)\in\Lambda_n^{\oplus rd}/\ker C_nC_{n-1}$ such that
\[
\col_T^{(n,1)}(z)\equiv C_{n-1}\cdot\col_T^{(n,2)}\mod\ker C_nC_{n-1}.
\]
On repeating this for $n$ times, we obtain the result.
\end{proof}

We shall show that the sequence $\left\{\col^{(n)}_T(z)\right\}_{n\ge1}$ gives us an element in $\Lambda^{\oplus rd}$. To do this, we need the following lemmas.

\begin{lemma}
The projection map $\pi_n$ induces a map on the quotients
\[
\pi_n':\Lambda_{n+1}^{\oplus r d}/\ker h_{n+1}\rightarrow \Lambda_{n}^{\oplus rd}/\ker h_{n}.
\]
\end{lemma}
\begin{proof}
Let $x\in\ker h_{n+1}$. Recall that 
\[
C_{n+1}\equiv \left(C_\vp\right)^{-1}\mod\omega_n,
\]
so we have{
\[
\pi_n(C_{n+1}\cdots C_1\cdot x)=\left(
\begin{array}{c|c}
I_{rd_0}&0\\ \hline
0&pI_{r(d-d_0)}
\end{array}\right)C^{-1}C_{n}\cdots C_1(\pi_n(x)).
\] }
Since $\Lambda_{n}$ has no $p$-torsion, we deduce that $\pi_n(x)\in\ker h_{n}$ as required.
\end{proof}

\begin{lemma}\label{lem:limit}
The inverse limit $\varprojlim_{\pi_n'} \left(\Lambda_n^{\oplus rd}/\ker h_n\right)$ is equal to $\Lambda^{\oplus rd}$.
\end{lemma}
\begin{proof}
The map $\pi_n'$ is surjective since $\pi_n$ is so. Hence, we have an isomorphism
\[
\varprojlim \Lambda_n^{\oplus rd}/\ker h_n\cong \Lambda^{\oplus rd}/\varprojlim\ker h_n.
\]
Indeed, if $x$ is an element of  $\Lambda^{\oplus r d}$ that lies inside $\varprojlim\ker h_n$, we have $M_T\cdot x=0$ as elements in $\calH^{\oplus rd}$. But $M_T$ has non-zero determinant, so $x=0$.
\end{proof}

\begin{theorem}\label{thm:decomposereg}
There exists a unique $\col^K_T(z)\in\Lambda^{\oplus r d}$ such that
\[
\begin{pmatrix}
\calL^K_{T,1}(z)\\
\vdots\\
\calL^K_{T,rd}(z)
\end{pmatrix}
=
M_T\cdot
\col^K_T(z).
\]
\end{theorem}
\begin{proof}
By Propositions~\ref{prop:integral} and \ref{prop:factorisation}, we have
\[
\begin{pmatrix}
\calL^K_{T,1}(z)\\
\vdots\\
\calL^K_{T,rd}(z)
\end{pmatrix}
\equiv
M_n\cdot\col_T^{(n)}(z)
\mod(\omega_n,\ker h_n).
\]
Recall from \cite[\S1.2]{perrinriou94} that if $F_1$ and $F_2$ are two elements of $\calH$ that are both $o(\log)$ and that $F_1\equiv F_2\mod\omega_n$ for all $n$, then $F_1=F_2$. Therefore, on letting $n\rightarrow\infty$, the theorem follows from Proposition~\ref{prop:matrix} and Lemma~\ref{lem:limit}.
\end{proof}

\begin{corollary}\label{cor:decompose} We have 
$\calL^K_T(z)=\begin{pmatrix}
v_1&\cdots& v_{rd}
\end{pmatrix}\cdot M_T\cdot\col^K_T(z)$.
\end{corollary}

Note that since $\calL^K_T$ is a $\Lambda$-homomorphism, the map
\begin{align*}
\HIw(K,T)&\rightarrow\Lambda^{\oplus rd}\\
z&\mapsto\col^K_T(z)
\end{align*}
is also a $\Lambda$-homomorphism.

\subsection{Dependence of the choice of basis}

Our construction of the Coleman map $\col^K_T$ depends on the choice of a basis of $\DD_K(T)$. In this section, we investigate this dependence. More precisely, let $v_1,\ldots, v_{rd}$ and $w_1,\ldots,w_{rd}$ be two bases of $\DD_K(T)$ that admit the construction of logarithmic matrices $M_{T,v}$ and $M_{T,w}$ respectively, as given in \S\ref{S:matrix}. Consequently, this results in two Coleman maps $\col^K_{T,v}$ and $\col^K_{T,w}$. We will study the relation between these two maps.

Let $B\in \GL_{rd}(\Zp)$ be the change of basis matrix satisfying
\begin{equation}
\begin{pmatrix}
v_1&\cdots&v_{rd}
\end{pmatrix}=
\begin{pmatrix}
w_1&\cdots&w_{rd}
\end{pmatrix}B.\label{eq:relation}
\end{equation}

\begin{lemma}\label{lem:changematrix}
The logarithmic matrices for the two bases are related by
\[
BM_{T,v}B^{-1}=M_{T,w}.
\]
\end{lemma}
­\begin{proof}
Let $C_{\vp,v}$ and $C_{vp,w}$ be the matrices of $\vp$ with respect to the bases $v_1,\ldots, v_{rd}$ and $w_1,\ldots,w_{rd}$ respectively. We decompose the two matrices {
\[
C_{\vp,v}= C_v\left(\begin{array}{c|c}
I_{rd_0}&0\\ \hline
0&\frac{1}{p}I_{r(d-d_0)}
\end{array}
\right)\quad\text{and}\quad C_{\vp,w}=C_w\left(\begin{array}{c|c}
I_{rd_0}&0\\ \hline
0&\frac{1}{p}I_{r(d-d_0)}
\end{array}
\right).
\] 
Since the action of $\vp$ on $\DD_K(T)$ is semi-linear, \eqref{eq:relation} implies that $BC_{\vp,v}B^{-1}=C_{\vp,w}$. Therefore, we have
\[
B\left(\begin{array}{c|c}
I_{rd_0}&0\\ \hline
0&pI_{r(d-d_0)}
\end{array}
\right)C_v^{-1}B^{-1}=\left(\begin{array}{c|c}
I_{rd_0}&0\\ \hline
0&pI_{r(d-d_0)}
\end{array}
\right)C_w^{-1}.
\]
Let $n\ge1$, it is known that $\Phi_{p^n}(\zeta)=p$ for $\zeta\in\mu_{p^m}\setminus \mu_{p^n}$, where $m>n$. Hence, 
\[
B\left(\begin{array}{c|c}
I_{rd_0}&0\\ \hline
0&\Phi_{p^n}(1+X)I_{r(d-d_0)}
\end{array}
\right)C_v^{-1}B^{-1}=\left(\begin{array}{c|c}
I_{rd_0}&0\\ \hline
0&\Phi_{p^n}(1+X)I_{r(d-d_0)}
\end{array}
\right)C_w^{-1}
\]}
as the two sides agree on infinitely many values of $X$. The lemma now follows from the construction of the logarithmic matrices.
\end{proof}

\begin{corollary}\label{cor:changeColeman}
The Coleman maps for the two bases are related by
\[
B\cdot\col^K_{T,v}=\col^K_{T,w}.
\]
\end{corollary}
\begin{proof}
We have by \ref{cor:decompose}
\[
\calL^K_T=\begin{pmatrix}
v_1&\cdots& v_{rd}
\end{pmatrix}\cdot M_{T,v}\cdot\col^K_{T,v}=\begin{pmatrix}
w_1&\cdots& w_{rd}
\end{pmatrix}\cdot M_{T,w}\cdot\col^K_{T,w}.
\]

On combining this with \eqref{eq:relation} and Lemma~\ref{lem:changematrix}, we have
\[
\begin{pmatrix}
w_1&\cdots& w_{rd}
\end{pmatrix} M_{T,w}B\cdot\col^K_{T,v}=\begin{pmatrix}
w_1&\cdots& w_{rd}
\end{pmatrix}\cdot M_{T,w}\cdot\col^K_{T,w}.
\]
Our result now follows from the linear independence of $w_1,\ldots,w_{rd}$ and the fact that $\det(M_{T,w})\ne0$.
\end{proof}

We now define Coleman maps for \textit{any} $\Zp$-basis $\{v_1,\ldots,v_{rd}\}$ of $\DD_K(T)$. Let $\{w_1,\ldots,w_{rd}\}$ be another basis that admits the construction of the logarithmic matrix $M_{T,w}$. Let $B\in \GL_{rd}(\Zp)$ be the change of basis matrix satisfying the same equation as \eqref{eq:relation}. We define
\begin{equation}
M_{T,v}:=B^{-1}M_{T,w}B\quad\text{and}\quad\col^K_{T,v}:=B^{-1}\cdot\col^K_{T,w}.\label{eq:newmatrix}
\end{equation}
Then Lemma \ref{lem:changematrix} and Corollary~\ref{cor:changeColeman} ensure that these objects are well-defined (i.e. independent of the choice of $w_1,\ldots,w_{rd}$. Furthermore, it is immediate from Corollary~\ref{cor:decompose} that we have the decomposition
\begin{equation}\label{eq:decomp0}
\calL^K_T(z)=\begin{pmatrix}
v_1&\cdots& v_{rd}
\end{pmatrix}\cdot M_{T,v}\cdot\col^K_{T,v}(z).
\end{equation}
Furthermore, if $\eta$ is a character modulo $p$, then Lemma~\ref{lem:evaluatematrix} implies that
\begin{equation}\label{eq:evaluatematrix}
\eta(M_{T,v})=C_{\vp,v},
\end{equation}
where $C_{\vp,v}$ is the matrix of $\vp$ with respect to $v_1,\ldots,v_{rd}$.

\subsection{Images of the Coleman maps}
\label{subsec:imagesofcol}

In this section, we will describe the images of the Coleman maps $\col^K_T$ (for a fixed basis of $\DD_K(T)$) at each isotypic component.

\subsubsection{Determinants of $\Lambda$-modules}

We first recall the definition of the determinant of a $\Zp[[X]]$-module as given in \cite[\S3.1.5]{perrinriou94}. If $M$ is a finitely generated projective $\Zp[[X]]$-module, $\det(M)$ is  the maximal exterior power of $M$. More generally, if $M$ is a finitely generated $\Zp[[X]]$-module that is not necessarily projective, let 
\[
0\rightarrow M_r\rightarrow \cdots \rightarrow M_1\rightarrow M_0\rightarrow M\rightarrow 0
\]
be a projective resolution, then $\det(M)$ is defined to be $\bigotimes_{i=0}^r\det(M_i)^{(-1)^i}$. This definition is independent of the choice of the projective resolution.

If $0\rightarrow M_1\rightarrow M_2\rightarrow M_3\rightarrow 0$ is a short exact sequence of $\Lambda$-modules, then
\[
\det(M_2)=\det(M_1)\otimes\det(M_2).
\]

For example, if $M=\Zp[[X]]/f\Zp[[X]]$ where $f\in\Zp[[X]]$, then by considering the exact sequence
\[
0\rightarrow f\Zp[[X]]\rightarrow \Zp[[X]]\rightarrow \Zp[[X]]\rightarrow 0,
\]
we see that $\det(M)=f^{-1}\Zp[[X]]$. More generally, if $M$ is a torsion $\Zp[[X]]$-module, we see that
\[
\Char_{\Zp[[X]]}M=\det(M)^{-1}.
\]

Let $M=(f_1,\ldots,f_r)$ be a $\Zp[[X]]$-submodule of $\Zp[[X]]^{\oplus r}$ such that $\Zp[[X]]^{\oplus r}/M$ is $\Zp[[X]]$-torsion. Write $f_i=(f_{i,j})_{j=1,\ldots,r}$ where $f_{i,j}\in\Zp[[X]]$, then $\det(M)$ is the $\Zp[[X]]$-module generated by the determinant of the $r\times r$ matrix whose entries are given by $f_{i,j}$.

More generally, if $M$ is a finitely generated $\Lambda$-module, we define $\det_\Lambda(M)$ to be
\[
\sum_{\eta}e_\eta\cdot \det(M^\eta)
\]
where the sum runs over all characters of $\Delta$. 

\subsubsection{Description of the images}

Let $\eta$ be a character modulo $p$. We shall describe the $\eta$-isotypic component of the image of the Coleman map $\col^K_T$.

\begin{lemma}\label{lem:evaluateColeman}
Let $z\in\HIw(K,T)$, then
\[
\eta\left(\col^K_{T,i}(z)\right)=
\begin{cases}\left[\exp^*_0(z),(1-\vp)^{-1}(p\vp-1)v_i'\right]
&\text{if $\eta$ is trivial,}\\
\frac{p}{\tau(\eta^{-1})}\left[\sum_{\sigma\in G_1}\theta^{-1}(\sigma)\exp^*_1(z^\sigma),v_i'\right]&\text{otherwise.}
\end{cases}
\]
for $i=1,\ldots,rd$.
\end{lemma}
\begin{proof}
By Lemma \ref{lem:evaluatematrix}, $\eta(M_T)=C_\vp$. So, Corollary~\ref{cor:decompose} implies that
\[
\eta(\calL_T^K(z))=\begin{pmatrix}
\vp(v_1)&\cdots&\vp(v_{rd})
\end{pmatrix}\cdot\col^K_T(z).
\]
When $\eta$ is trivial, Lemma~\ref{lem:regulatorformula} implies that
\[
\begin{pmatrix}
v_1&\cdots&v_{rd}
\end{pmatrix}\cdot\eta\left(\col^K_T(z)\right)=
\sum_{i=1}^{rd}\left[\exp^*_0(z),v_i'\right](1-\vp)(1-p^{-1}\vp^{-1})^{-1}\vp^{-1}(v_i).
\]
Since $\vp$ and $p^{-1}\vp^{-1}$ are the adjoints of each other under $[\sim,\sim]$, the right-hand side can be rewritten as
\[
\sum_{i=1}^{rd}\left[\exp^*_0(z),(1-\vp)^{-1}(p\vp-1)v_i'\right]v_i.
\]

When $\eta$ is non-trivial, Lemma~\ref{lem:regulatorformula} implies that
\[
\begin{pmatrix}
v_1&\cdots&v_{rd}
\end{pmatrix}\cdot\eta\left(\col^K_T(z)\right)=
\sum_{i=1}^{rd}\frac{p}{\tau(\eta^{-1})}\left[\sum_{\sigma\in G_1}\theta^{-1}(\sigma)\exp^*_1(z^\sigma),v_i'\right]v_i.
\]
Hence the result.
\end{proof}

\begin{lemma}\label{lem:evaluate}
Let $a_1,\ldots, a_{rd}\in\Zp$. We have $\sum_{i=1}^{rd}a_ie_\eta\col_{T,i}^K(z)$ equal to $0$ when evaluated at $X=0$ if either
$\eta$ is the trivial character and 
\[
\sum_{i=1}^{rd}a_i(1-\vp)^{-1}(1-p\vp)v_i'\in\Fil^0\DD_K(T^*(1)),
\]
or $\eta$ is non-trivial and 
\[
\sum_{i=1}^{rd}a_iv_i'\in\Fil^0\DD_K(T^*(1)),
\]
\end{lemma}
\begin{proof}
We remark that $\eta(F)=e_\eta\cdot F|_{X=0}$ for any element $F\in\calH$ and
\[
\left[\exp^*(z),w\right]=0
\]
for all $w\in\Fil^0\DD_K(T^*(1))$ and $z\in H^1(F_n,T)$ where $n\ge0$. Therefore, our result follows from Lemma~\ref{lem:evaluateColeman}.
\end{proof}

{
We define two $\Qp$-linear maps 
$\cA,\cB:\Qp^{\oplus rd}\rightarrow \DD_K(T^*(1))/\Fil^0\DD_K(T^*(1))\otimes\Qp$ by setting
\begin{align*}
(a_1,\ldots,a_{rd})\mapsto &\sum_{i=1}^{rd}a_i(1-\vp)^{-1}(1-p\vp)v_i'\mod\Fil^0\DD_K(T^*(1)),\\
(a_1,\ldots,a_{rd})\mapsto&\sum_{i=1}^{rd}a_iv_i'\mod\Fil^0\DD_K(T^*(1)).
\end{align*}
We have the dual maps $\cA^*,\cB^*:\Fil^0\DD_K(T)\otimes\Qp\rightarrow \Qp^{\oplus rd}$ given by
\begin{align*}
v\mapsto &(1-\vp)\left(1-\frac{\vp}{p}\right)^{-1}v\\
v\mapsto &v
\end{align*}
on identifying $\Qp^{\oplus rd}$ with $\DD_K(T)\otimes\Qp$ via the basis $v_1,\ldots,v_{rd}$.
\begin{corollary}\label{cor:image}
If $\eta$ is trivial, then $\image\left(\col^K_{T}\right)^\eta$ is contained in 
\[
\left\{F\in\Zp[[X]]^{\oplus r d}: F(0)\in\image(\mathcal{A^*}) \right\}.
\]
If $\eta$ is non-trivial, then $\image\left(\col^K_T\right)^\eta$ is contained in
\[
\left\{F\in\Zp[[X]]^{\oplus r d}: F(0)\in\image(\mathcal{B^*}) \right\}.
\]
\end{corollary}
\begin{proof}
Lemma~\ref{lem:evaluate} tells us that if $F\in\image\left(\col^K_{T}\right)^\eta$, then $F(0)\in\ker(\cA)^\perp$ (respectively $F(0)\in\ker(\cB)^\perp$), where $\perp$ denotes the orthogonal complement under the pairing
\begin{align*}
\Zp^{\oplus rd}\times\Zp^{\oplus rd}&\rightarrow \Zp\\
((a_1,\ldots,a_{rd}),(b_1,\ldots,b_{rd}))&\mapsto\sum_{i=1}^{rd} a_ib_i.
\end{align*}
Hence the result by duality.
\end{proof}
\begin{proposition}\label{prop:image}
The containments in Corollary~\ref{cor:image} have finite index.
\end{proposition}
\begin{proof} By the Colmez-Perrin-Riou reciprocity law, with respect to a $\Lambda$-basis of $\HIw(\Qp,T_p(A))$ and a $\Zp$-basis of $\DD_K(T)$, the determinant of $\calL_T$ is, up to a unit in $\Lambda$, $(\log(1+X)/p)^{r(d-d_0)}$. By Proposition~\ref{prop:matrix}, the determinant of $M_T$ is, up to a constant in $\Zp^\times$,  $(\log(1+X)/pX)^{r(d-d_0)}$. Therefore, 
\[
\det{}_\Lambda\left(\image\left(\col^K_T\right)\right)=X^{r(d-d_0)}\Lambda.
\]
by Corollary~\ref{cor:decompose}.
Note that $\cA$ and $\cB$ are surjective and that $\Fil^0\DD(T^*(1))$ has rank $r(d-d_0)$ over $\Zp$. Thus $\image(\cA^*)$ and $\image(\cB^*)$ have rank $rd_0$ and the modules described in Corollary~\ref{cor:image} have determinant $X^{r(d-d_0)}$, the quotients of the containments have trivial determinant.\end{proof}
\begin{proposition}\label{prop:iageofcolemanis}
Let $I\subset\{1,\ldots,rd\}$ be a subset of cardinality $k$. Let $\eta$ be a Dirichlet character modulo $p$. Define $\pr_I$ be the projection $\sum_{i=1}^{rd} a_iv_i\mapsto\sum_{i\in I}a_iv_i$  and define 
 $$U_I^\eta:=\begin{cases}\pr_I\left((1-\vp)\left(1-\frac{\vp}{p}\right)^{-1}\Fil^0\DD_K(T)\right)\,,&\text{if $\eta$ is trival,}
 \\ \pr_I\left(\Fil^0\DD_K(T)\right)\,, &\text{otherwise.}
 \end{cases}$$
  Then, $\image\left(\oplus_{i\in I}\col^K_{T,i}\right)^\eta$ is contained inside $$\left\{F\in\oplus_{i\in I}\Zp[[X]]:F(0)\in U_I^\eta\right\},$$ 
  if we identify $\DD_K(T)$ with $\Zp^{\oplus rd}$ via our choice of basis. Furthermore, the containment is of finite index.
\end{proposition}
\begin{proof}
We assume that $\eta$ is the trivial character in this proof. The other case can be proved similarly. Let  $\tilde{\pr}_I:\Zp^{\oplus rd}\rightarrow\oplus_{i\in I}\Zp$ be the natural projection. Then by Corollary~\ref{cor:image}, $\image\left(\oplus_{i\in I}\col^K_{T,i}\right)^\eta$ is contained in 
\[
\left\{F\in\oplus_{i\in I}\Zp[[X]]: F(0)\in\tilde{\pr}_I\left(\image(\cA^*)\right)\right\}.
\]
with finite index. Hence the result by the description of $\cA^*$.
\end{proof}
\begin{corollary}\label{cor:pseudosurj2}
If $I$ and $\eta$ as above, then  $\image\left(\oplus_{i\in I}\col^K_{T,i}\right)^\eta$ is contained in a free $\Zp[[X]]$-module, with finite index.
\end{corollary}
\begin{proof}
Note that $U_I^\eta$ is a saturated $\Zp$-module inside $\oplus_{i\in I}\Zp$, so there exists a $\Zp$-basis $u_1,\ldots,u_{k}$ of $\oplus_{i\in I}\Zp$ such that $u_1,\ldots,u_m$ generates $U_I^\eta$ for some integer $m$. Consider  $u_1X,\ldots, u_mX, u_{m+1},\ldots, u_{k}$ as elements of $\oplus_{i\in I}\Zp[[X]]$. 
By Nakayama's lemma, these elements form a $\Zp[[X]]$-basis of $\left\{F\in\oplus_{i\in I}\Zp[[X]]:F(0)\in U_I^\eta\right\}$.
\end{proof}
\begin{corollary}\label{cor:pseudosurj}
Let $I\subset\{1,\ldots,rd\}$ be a subset of cardinality $k$. 
\begin{itemize}
\item[(a)] Let $\eta$ be the trivial character. The index of $\image\left(\oplus_{i\in I}\col^K_{T,i}\right)^\eta$ inside $\Zp[[X]]^{\oplus k}$ is finite if and only if $${\rm span}((1-\vp)^{-1}(p\vp-1)v_i':i\in I)\cap\Fil^0\DD_K(T^*(1))=0;$$
\item[(b)] Let $\eta$ be a Dirichlet character of conductor $p$. The index of $\image\left(\oplus_{i\in I}\col^K_{T,i}\right)^\eta$ inside $\Zp[[X]]^{\oplus k}$ is finite if and only if $${\rm span}(v_i':i\in I)\cap\Fil^0\DD_K(T^*(1))=0.$$
\end{itemize} 
\end{corollary}
\begin{proof}
We prove (a) only. The set $U_I^\eta$ in the statement of Proposition~\ref{prop:iageofcolemanis} is $\oplus_{i\in I}\Zp$ if and only if
\[
(1-\vp)\left(1-\frac{\vp}{p}\right)^{-1}\Fil^0\DD_K(T)+{\rm span}(v_i:i\notin I)=\DD_K(T).
\]
Therefore, on taking orthogonal complements, this is equivalent to
\[
{\rm span}((1-\vp)^{-1}(p\vp-1)v_i':i\in I)\cap\Fil^0\DD_K(T^*(1))=0\]
as we have the elementary formula $(U+ V)^\perp=U^\perp+V^\perp$.
\end{proof}
}



\section{Conjectures}

Let $F$ be a number field of degree $r$ where the prime $p$ is unramified. We assume that $F$ is either  a totally real field or a CM field. We fix a rank $d$ continuous $\Zp$-representation $T$ of $G_F$ such that $T$ verifies the hypotheses (H.F.-L.), (H.S.), (H.Leop) and (H.nA) introduced above.

Furthermore, in order to simplify notation, we set $g=[F:\QQ]\times d$ and define $g_{+}:=\dim \, \left(\Ind_{F/{\QQ}}\,T\otimes\Qp\right)^+$ as above. Set $g_-=g-g_+$ and suppose throughout that $g_->0$. Let $\DD_{p}(T)$ be the direct sum $\oplus_{\fp|p}\DD_{F_\fp}(T)$. We assume until the end that the following form of the \emph{Panchishkin condition} holds true:
\begin{itemize}
\item[(H.P.)] $\dim \,\left( \textup{Fil}^0 \DD_p(T)\otimes\QQ_p\right)\,=g_-$\,.
\end{itemize}

Let $S$ be the set of primes of $F$ where $T$ is ramified and those that divide $p$. If $L$ is an extension of $F$, we write $G_{L,S}$ for the Galois group of the maximal extension of $L$ unramified outside $S$. Fix until the end an even Dirichlet character $\eta$ of $\Delta=\Gal(\QQ(\mu_p)/\QQ)$.

For $i=1,2$, we define
\[
H^i_{\Iw,S}(F,T)=\varprojlim H^i(G_{F(\mu_{p^n}),S},T).
\]
By \cite[Proposition~1.3.2]{perrinriou95}, our assumptions on $T$ imply that at each isotypic component,  $H^2_{\Iw,S}(F,T)$ is $\Zp[[X]]$-torsion and $H^1_{\Iw,S}(F,T)_\pm$ is of rank $g_\mp$ over $\Lambda_\pm$. Let $\ff_2\in\Lambda$ be the characteristic ideal of $H^2_{\Iw,S}(F,T)$. We write $\loc$ for the localization map
\[
\loc: H^1_{\Iw,S}(F,T)\longrightarrow\HIw(F_p,T):=\bigoplus_{\fp|p}H^1_{\Iw}(F_\fp,T),
\]
and also for the map induced on the $\eta$-isotypic submodule.
\subsection{Semi-local decomposition}

Consider the map
\[
\calL_T^F=\oplus_{\fp|p}\calL_{T}^{F_\fp}:H^1_{\Iw}(F_p,T)\longrightarrow\calH\otimes_{\Zp}\DD_{p}(T).
\]

We fix a basis $v_1,\ldots,v_{g}$ for $\DD_p(T)$ consisting of a sub-basis $\{v_{\fp,i}\}$ of $\DD_{F_\fp}(T)$ for each $\fp|p$. Let $M_T$ be the $g\times g$ block diagonal matrix where the entries are given by $M_{T|G_{F_\fp}}$  for $\fp|p$, where $M_{T|G_{F_\fp}}$ is the logarithmic matrix as constructed in \eqref{eq:newmatrix}. We write $(\col_{T,i})_{i=1}^g$ for the column vector given by $\left(\col_T^{F_\fp}\right)_{\fp|p}$. Then, \eqref{eq:decomp0} gives us the decomposition of $\Lambda$-homomorphism
\begin{equation}
\calL_T^F=\begin{pmatrix}
v_1&\cdots&v_{g}
\end{pmatrix}
\cdot M_T\cdot \begin{pmatrix}
\col_{T,1}\\
\vdots\\
\col_{T,g}
\end{pmatrix}\label{eq:decomp}
\end{equation}
for some block diagonal matrix $M_T\in M_{g\times g}(\calH)$, whose entries are all $o(\log)$.

Let $\loc_p$ be the localization from $H^1_{\Iw,S}(F,T)$ to $H^1_{\Iw}(F_p,T)$. We write $\calL_{\loc}$ for the composition $\calL_T^F\circ\loc$.

\begin{defn}
We write $\fI_p$ for the set of tuples $\uI=(I_\fp)_{\fp|p}$ where each $I_{\fp}$ is a subset of of $\{1,\ldots,[F_\fp:\Qp]d\}$ such that $\sum\#I_{\fp}=g_-$. This can be equally regarded as the set of subsets of $\{1,\ldots, g\}$ of size $g_-$. We shall construct a Selmer group for each $\uI\in\fI_p$, which we conjecture to be $\Lambda$-cotorsion.
\end{defn}
\subsection{Perrin-Riou's main conjecture}


\begin{defn}\label{def:admissiblebasis}
{Let $\frak{B}=\{v_1,\cdots,v_{g}\}$ be a $\ZZ_p$-basis of $\DD_p(T)$. Let $\frak{B}^\prime=\{v_1^\prime,\cdots,v_{g}^\prime\}\subset\DD_p(T^*(1))$ be its dual basis. The basis $\frak{B}$ is called \emph{admissible} if for any $\uI \in \fI_p$, we have 
\begin{equation}\label{eqn:defineadmis} \textup{span}\left(v_i^\prime: i\in \uI \right)\cap \Fil^0\DD_p(T^*(1)) = 0\,
\end{equation}
and \emph{strongly admissible} if in addition to (\ref{eqn:defineadmis}) we have
$$\textup{span}\left((1-\vp)^{-1}(p\vp-1)v_i^\prime: i\in \uI \right)\cap \Fil^0\DD_p(T^*(1)) = 0\,.
$$}
\end{defn}

\begin{proposition}
\label{prop:admissiblebasisexists}
A strongly admissible basis exists.
\end{proposition}
The proof Proposition~\ref{prop:admissiblebasisexists} will be given in Appendix~\ref{appendix:linearalgebra}. 
{\begin{remark}
\label{rem:whyadmissible}
We note that the strong admissibility condition would allow us to apply Proposition~\ref{prop:iageofcolemanis}  and conclude as in Corollary~\ref{cor:pseudosurj} that the signed Coleman maps we shall be using are pseudo-surjective onto a free $\ZZ_p[[X]]$-module. 
\end{remark}}

{For $\uI\in\fI_p$, let $N_{\uI}$ be the $\Zp$-submodule generated by the sub-basis $\{v_i':i\in\uI\}$. Perrin-Riou in \cite{perrinriou92} associates $N_{\uI}$ a height pairing  $\langle,\rangle_{N_{\uI}}$.  Since we have $N_{\uI}\cap\Fil^0\DD_p(T^*(1))=0$ for $\uI\in\fI_p$, the submodule $N_{\uI}$ is regular in the sense of \cite[\S3.1.2]{perrinriou95} if and only if the height pairing $\langle,\rangle_{N_{\uI}}$ is non-degenerate (see also \cite[\S2.1]{benois14}). }

\begin{defn}
\label{def:padicperiods}
{For the dual motive $\mathcal{M}^*(1)$ to $\mathcal{M}$, we let $\Omega_{\mathcal{M}^*(1),p}(\uI)$ denote Perrin-Riou's $p$-adic period (given as in \cite{perrinriou95}) associated to the determinant of $\langle,\rangle_{N_{\uI}}$. When $N_{\uI}$ is not a regular subspace, this period shall be set to be zero.}
\end{defn}


\begin{conj}\label{conj:TpadicL}
There exists an analytic $p$-adic $L$-function 
$$L_p(\mathcal{M}^*(1))\in \calH_+\otimes\wedge^{g_{-}}\DD_p(T)$$ such that for all even Dirichlet characters $\theta$ of conductor $p^n>1$, we have
\begin{align*}
\theta&\left(L_p(\mathcal{M}^*(1))\right)=\\&\sum_{\uI\in\fI_p}\left(\frac{p^n}{\tau(\theta^{-1})}\right)^{g_{-}}{L_{\{p\}}(\mathcal{M}^*(1),\theta^{-1},1)}\frac{\Omega_{\mathcal{M}(\theta)^*(1),p}(\uI)}{{\Omega_{\mathcal{M}(\theta)^*(1)}(\uI)}}\cdot \vp^{n}\left(\wedge_{i\in \uI}\,v_{i}\right)\,.
\end{align*}
When $\theta$ is the trivial character,
\begin{equation}\label{eq:trivialinterpolation}
\theta\left(L_p(\mathcal{M}^*(1))\right)=\sum_{\uI\in\fI_p}{L_{\{p\}}(\mathcal{M}^*(1),1)}\frac{\Omega_{\mathcal{M}^*(1),p}(\uI)}{\Omega_{\mathcal{M}^*(1)}(\uI)}\cdot (1-\vp)(1-p^{-1}\vp^{-1})^{-1}\left(\wedge_{i\in \uI}v_{i}\right) \, .
\end{equation}
Here, $L_{\{p\}}$ denotes the $L$-function with the Euler factors at $p$ removed.
\end{conj}
{Above $\Omega_{\mathcal{M}(\theta)^*(1)}(\uI)$ is Deligne's period so that the quotient $\frac{L_{\{p\}}(\mathcal{M}^*(1),\theta^{-1},1)}{\Omega_{\mathcal{M}(\theta)^*(1)}(\uI)}$ is an algebraic number.  Fixing an embedding $\overline{\QQ} \hookrightarrow \overline{\QQ}_p$ we regard this as an element of $\overline{\QQ}_p$. We also implicitly assert as part of the conjecture above that there is a choice of a normalization of Deligne's period amenable to $p$-adic interpolation.}

\begin{remark} Our interpolation formulae are not quite the ones stated in \cite[\S4.2]{perrinriou95} that predict a relation between the leading term of the $p$-adic $L$-function and complex $L$-values. Rather, we opt for a formulation that is closer to the existing one for elliptic curves and the one stated in \cite{coatesPR}. 
\end{remark}

The main conjecture of Perrin-Riou relates this conjectural $p$-adic $L$-function to the following module.

\begin{defn}Perrin-Riou's module of $p$-adic $L$-function is defined to be
\[
\Iarith(T)=\det{}_{\Lambda}\left(\image(\calL_{\loc})\right)\otimes\det{}_{\Lambda}\left(H^2_{\Iw,S}(F,T)\right)^{-1}.
\]
\end{defn}

\begin{conj}[Perrin Riou's Main Conjecture]
\label{TCP}As $\Lambda_+$-modules, we have
\[L_p(\mathcal{M}^*(1))\Lambda_+=\Iarith(T)_+.\]
\end{conj}

We now study the conjectural $p$-adic $L$-function $L_p(\mathcal{M}^*(1))$ further and relate it to the regulator map of Perrin-Riou via the Kolyvagin systems we construct in the appendix.

\begin{defn}
\label{define:thelevelnregulator}
Let $\Xi=\xi_1\wedge\cdots\wedge\xi_{g_-}\in\wedge^{g_-} H^1_{\Iw,S}(F,T)_+$ be any element and let $\theta$ be an even Dirichlet character of conductor $p^n$. For $\uI=(I_\fp)_{\fp|p}\in\fI_p$, we define
$$\mathfrak{M}_\theta^{\uI}(\Xi):=\displaystyle{\left(\left[\sum_{\sigma\in G_n}\theta(\sigma)\exp_n^*\left(\loc_\fp(\xi_i)^\sigma\right),v_{\fp,j}'\right]\right)_{1\le i\le g_-\,,\,j\in I_\fp}}.$$
\end{defn}


Let $\frak{K}(T)$ denote the $\Lambda$-module of \emph{Kolyvagin determinants}, given as in Definition~\ref{define:theKolyvaginconstructeddeterminantelement}(ii).
\begin{conj}
\label{conj:Tspecialelement}
There exists a (unique) non-zero element $\cc=\cc_1\wedge\cdots\wedge\cc_{g_-} \in\frak{K}(T)$ such that
\[
\det\left(\mathfrak{M}_\theta^{\uI}(\cc)\right)={L_{\{p\}}(\mathcal{M}^*(1),\theta^{-1},1)}\frac{\Omega_{\mathcal{M}(\theta)^*(1),p}(\uI)}{\Omega_{\mathcal{M}(\theta)^*(1)}(\uI)},
\]
for all $\uI$ and $\theta$ as in Definition~\ref{define:thelevelnregulator}.
\end{conj}
We will refer to this conjecture as the \emph{reciprocity conjecture for Kolyvagin-determinants}.
\begin{proposition}\label{prop:Twedgeinterpolation}
For $\cc\in\frak{K}(T)$ verifying Conjecture~\ref{conj:Tspecialelement},  $\calL_{\loc}(\cc)$ satisfies the interpolation properties given in Conjecture~\ref{conj:TpadicL}.
\end{proposition}
\begin{proof}
This follows from Lemmas~\ref{lem:regulatorformula} and \ref{lem:wedge}.
\end{proof}

\begin{remark}
Note that we have only considered the interpolation problem for the twists of the motive $\mathcal{M}^*(1)$ by even characters $\theta$ of $\Gamma$. One can also formulate a conjecture for $\emph{odd}$ characters, for which one needs to replace everywhere $g_-$ by $g_+$ (and vice-versa). Note also that upon studying the motive $\mathcal{M}\otimes\omega$ (where $\omega$ is the Teichm\"uller character) in place of $\mathcal{M}$, one may reduce the consideration for odd characters to the case of even characters.
\end{remark}

If $M$ is a $\Lambda$-module such that $M^\eta$ is $\Zp[[X]]$-torsion for all even  characters of $\eta$, we define the characteristic ideal 
\[
\Char_{\Lambda_+} M_+:=\sum_{\eta}e_{\eta}\cdot\Char_{\Zp[[X]]}M^\eta
\] 
where the sum runs over all even characters of $\Delta$.

\begin{proposition}\label{propn:TCP}
If Conjecture~\ref{conj:Tspecialelement} holds, then Conjecture~\ref{TCP} is equivalent to the assertion that
\begin{equation}\label{reformulateTCP}
\Char_{\Lambda_+}\left(H^2_{\Iw,S}(F,T)_+\right)=\Char_{\Lambda_+}\left(H^1_{\Iw,S}(F,T)_+/(\cc_1,\ldots,\cc_{g_-})\right).
\end{equation}
\end{proposition}
\begin{proof}
For any non-zero element $\cc=\cc_1\wedge\cdots\wedge\cc_{g_-}\in\wedge^{g_-}H^1_{\Iw,S}(F,T)_+$, we write $\ff_\cc$ for a generator of $\Char_{\Lambda_+}H^1_{\Iw,S}(F,T)_+/(\cc_1,\ldots,\cc_{g_-})$. 
Therefore, we have
\[
e_+\cdot\Iarith(A)=\ff_2\ff_\cc^{-1}\cdot\calL_{\loc}(\cc)\cdot\Lambda_+
\]
for any non-trivial $\cc$. If furthermore 
$$\calL_{\loc}(\cc)=L_p(\mathcal{M}^*(1)),$$ 
the result follows immediately.
\end{proof}




\subsection{Bounded $p$-adic $L$-functions}

Throughout, we assume that Conjecture \ref{conj:Tspecialelement} holds. 
Let $\cc=\cc_1\wedge\cdots\wedge\cc_{g_-}$ be the element verifying the conjecture. 

\begin{defn}
\label{define:TmodifiedSelmerColeman}
For $\uI\in\fI_p$, we define 
\begin{align*}
\col_T^{\uI}:\HIw(F_p,T)&\rightarrow\Lambda^{\oplus g_-}\\
z&\mapsto\oplus_{i\in\uI}\col_{T,i}(z)
\end{align*}
and $H^1_{\uI}(F_p,T)$ is defined to be the kernel of $\col_T^{\uI}$.
\end{defn}
 
\begin{lemma}\label{lem:Timage}
For $\uI\in \fI_p$ and a character $\eta$ modulo $p$, there exists an integer $n(\uI,\eta)\ge 0$ such that 
$$\det\left(\image\left(\col_T^{\uI}\right)^\eta\right)=X^{n(\uI,\eta)}\Zp[[X]].$$
If the basis of $\DD_p(T)$ that determines $\col_T^{\uI,\eta}$ as in \eqref{eq:decomp} is strongly admissible in the sense of Definition~\ref{def:admissiblebasis}, then we may take $n(\uI,\eta)=0$.
\end{lemma}
\begin{proof}
{This follows from Proposition~\ref{prop:iageofcolemanis} and Corollary~\ref{cor:pseudosurj}.}
\end{proof}
To simplify notation we sometimes will write $\col_T^{\uI}(\cc_i)$ in place of $\col_T^{\uI}(\loc(\cc_i))$ for $1\le i\le g_-$.
\begin{defn}
\label{def:signedpadicLfunc}
For each $\uI\in \fI_p$, we define the $p$-adic $L$-function $L_{\uI}(\mathcal{M}^*(1))$ to be $\det\left(\col^{\uI}_T(\cc_i)\right)$. 
\end{defn}

\begin{lemma}\label{lem:Tdet}
We have
\begin{align*}
\det\left({\image\left(\col_T^{\uI}\right)}\Big{/}{\textup{span}_{\Lambda}\,\left\{\col^{\uI}_T(\cc_i)\right\}_{i=1}^{g_-}}\right)^\eta=\left(L_{\uI}(\mathcal{M}^*(1))^\eta/X^{n(\uI,\eta)}\right)\cdot\Zp[[X]]
\end{align*}
for some integer $n(\uI,\eta)\ge0$.
If the basis of $\DD_p(T)$ that determines $\col_T^{\uI,\eta}$ is strongly admissible then we may take $n(\uI,\eta)=0$.

\end{lemma}
\begin{proof}
This follows at once from Lemma~\ref{lem:Timage} using the fact that taking $\det$ is compatible with exact sequences. 
\end{proof}

The following results explain how these functions are related to complex $L$-values and Perrin-Riou's $p$-adic $L$-functions.

\begin{proposition}
Let $\mathcal{C}$ be the matrix of $(1-\vp)^{-1}(p\vp-1)$ with respect to the basis $v_1',\ldots,v_g'$.
Let $\eta$ be a character on $\Delta$ and $\uI\in\fI_p$, then
\[
\eta\left(L_{\uI}(\mathcal{M}^*(1))\right)=
\begin{cases}
{L_{\{p\}}(\mathcal{M}^*(1),1)}\sum_{\uJ\in\fI_p}\mathcal{C}_{\uI,\uJ}\frac{\Omega_{\mathcal{M}^*(1),p}(\uJ)}{\Omega_{\mathcal{M}^*(1)}(\uJ)}
&\text{if $\eta$ is trivial,}\\
\left(\frac{p^n}{\tau(\eta^{-1})}\right)^{g_{-}}{L_{\{p\}}(\mathcal{M}^*(1),\eta^{-1},1)}\frac{\Omega_{\mathcal{M}(\eta)^*(1),p}(\uI)}{{\Omega_{\mathcal{M}(\eta)^*(1)}(\uI)}}&\text{otherwise,}
\end{cases}
\]
where $\mathcal{C}_{\uI,\uJ}$ is the determinant of the $g_-\times g_-$ submatrix  of $\mathcal{C}$ whose entries correspond to the elements of $\uI$ and $\uJ$.
\end{proposition}
\begin{proof}
When $\eta$ is trivial, we have from Lemma~\ref{lem:evaluateColeman} the formula
\[
\eta\left(L_{\uI}(\mathcal{M}^*(1))\right)=\det\left(\left[\exp^*_0(\xi_i),(1-\vp)^{-1}(p\vp-1)v_{\fp,j}'\right]\right).
\]
So, we may expand $(1-\vp)^{-1}(p\vp-1)$ by the matrix $\mathcal{C}$ and obtain the first part of the proposition using Definition~\ref{define:thelevelnregulator}.

When $\eta$ is non-trivial, this follows immediately from Lemma~\ref{lem:evaluateColeman} and Definition~\ref{define:thelevelnregulator}.
\end{proof}

\begin{lemma}\label{lem:wedge}
Let $R$ be a commutative ring. Let $M$ and $M'$  be two $R$-modules, with a homomorphism $F:M\rightarrow M'$ of $\Lambda$-modules. Let $m\le n$ be integers. Fix $a_1,\ldots, a_m\in M$ and $b_1,\ldots, b_n\in M'$ with
\[
F(a_i)=\sum_{j=1}^n r_{i,j}b_j
\]
for $i=1,\ldots,m$. Then
\[
F(a_{1}\wedge\ldots\wedge a_{m})=\sum_{j_1<\cdots<j_m}\det(r_{j_1,\ldots,j_m})\,b_{j_1}\wedge\cdots\wedge b_{j_m}
\]
where $r_{j_1,\ldots,j_m}$ is the $m\times m$ matrix whose $(k,l)$-entry is given by $r_{k,j_l}$.
\end{lemma}
\begin{proof}
This is standard multi-linear algebra.
\end{proof}

\begin{theorem}
\label{thm:decomposepadicL} For $\uI,\uJ\in\fI_p$, let $M_T^{\uI,\uJ}$ be the $g_-\times g_-$ the submatrix of $M_T$ whose entries correspond to the elements of $\uI$ and $\uJ$. Then there is a decomposition
\[
L_p(\mathcal{M}^*(1))=\sum_{\uI,\uJ\in\fI_p}\wedge_{i\in \uI}v_{i}\det(M_T^{\uI,\uJ})L_{\uJ}(\mathcal{M}^*(1)).
\]
\end{theorem}
\begin{proof}
Let the $(j,k)$-entry of $M_T$ be $m_{j,k}$. Recall that \eqref{eq:decomp} says that
\[
\calL_{\loc}(\cc_i)=\sum_{1\le j,k\le g} v_j m_{j,k}\col_{T,k}(\cc_i)
\]
for $1\le i\le g_-$. Hence by Lemma~\ref{lem:wedge}, we have
\begin{align*}
\wedge_{i=1}^{g_-}\calL_{\loc}(\cc_i)&=\sum_{\uI\in\fI_p}\wedge_{i\in \uI}v_{i}\det\left(\sum_{k=1}^{g} m_{j,k}\col_{T,k}(\cc_i)\right)_{j\in \uI,1\le i\le g_-}\cdot\\
&=\sum_{\uI\in\fI_p}\wedge_{i\in \uI}v_{i}\sum_{\uJ\in\fI_p}\det( m_{j,k})_{j\in \uI,k\in \uJ}\cdot\det(\col_{T,k}(\cc_i))_{k\in \uJ,1\le i\le g_-}
\end{align*}
as required.
\end{proof}

\subsection{Modified Selmer groups}\label{subset:modifiedSelmergroups}
We now define modified Selmer groups using the Coleman maps $\col_T^{\uI}$.
\begin{lemma}
For any subset $\{i_1,\ldots,i_k\}$ of the set $\{1,\ldots,g\}$, the $\Lambda$-module ${\bigcap_{j=1}^k\ker\col_{i_j}}$ is torsion-free of rank $g-k$.
\end{lemma}
\begin{proof}
Recall that the $\Lambda$-torsion submodule of $\HIw(F_p,T)$ is isomorphic to the module $H^0(F(\mu_{p^\infty})_p,T)$, which is trivial since we assumed (H.nA). It follows that the $\Lambda$-module $\HIw(F_p,T)$ is torsion-free.

By Proposition~\ref{prop:image}, $\image\left(\oplus_{j=1}^k\col_{i_j}\right)$ has rank $k$ over $\Lambda$. But  $\HIw(F_p,T)$ is of rank $g$ over $\Lambda$ thus $\ker\left(\oplus_{j=1}^k\col_{i_j}\right)=\bigcap_{j=1}^k\ker\col_{i_j}$ has rank $g-k$ over $\Lambda$.
\end{proof}

\begin{corollary}
\label{cor:intersectionofcolemanmaps}
\begin{itemize}
\item[(a)]For each $\uI\in\fI_p$, the torsion-free $\Lambda$-module $H^1_{\uI}(F_p,T)$ has rank $g_{+}$.
\item[(b)]$\displaystyle{\bigcap_{i=1}^{g}\ker\col_i=0}$.
\end{itemize}
\end{corollary}

\begin{lemma}
\label{lemma:extendto2g}
Let $W$ be (a torsion-free) $\Lambda$-submodule of $\HIw(F_p,T)$ generated by at most $g_{-}$ elements. Then there is an $\uI\in \fI_p$ such that 
$$W\cap H^1_{\uI}(F_p,T)=0\,.$$
\end{lemma}
\begin{proof}
Assume contrary that 
$$W\cap H^1_{\uI}(F_p,T)\neq0$$
for any $\uI\in \fI_p$. We prove by induction on $0\leq k\leq g_+$ that for every subset $J$ of $\{1,\cdots,g\}$ of size $g_-+k$, there is an non-zero element 
$$0\neq w_J \in W\cap \left(\bigcap_{i\in J}\ker\col_i\right).$$ When $k=0$, this is the hypothesis of the lemma. Assume its truth for $k=n<g_+$ and consider $J=\{i_1,\cdots,i_{g_-+n+1}\}\subset \{1,\cdots,g\}$. Set $J_s=J\backslash \{i_s\}$ for $s=1,\cdots, g_-+n+1$ and choose using the induction hypothesis a non-zero element $z_s \in W\cap \left(\bigcap_{i\in J_s}\ker\col_i\right)$. As the $\Lambda$-module $W$ is generated by at most $g_-$ elements, it follows that $\{z_s\}_{s=1}^{g_-+n+1}$ verifies a non-trivial relation
$$b_1z_1+b_2z_2+\cdots+b_{g_-+n+1}z_{g_-+n+1}=0,$$
where $b_i \in\Lambda$.  Let $s_0 \in\{1,\cdots,g_-+n+1\}$ be the smallest index such that $b_{s_0}\neq 0$. Then observe that $b_{s_0}z_{s_0}$ is non-zero since $W$ is torsion free and $b_{s_0}z_{s_0}\in \textup{span}\{z_i\}_{i\neq s_0} \subset \ker \col_{i_{s_0}}$, where the latter containment is due to our choice of the elements $z_j$'s. On the other hand, $b_{s_0}z_{s_0} \in \bigcap_{s\neq s_0}\ker \col_{i_s} $ by the choice of $z_{s_0}$, hence 
$$0\neq b_{s_0}z_{s_0}\in \ker \col_{{s_0}} \cap \left( \bigcap_{s\neq s_0}\ker \col_{i_s} \right)=\bigcap_{i\in J}\ker\col_i \,,$$
as desired. Now this shows (for $k=g_+$) that 
$$W\cap \left(\bigcap_{i=1}^{g}\ker\col_i\right)\neq0\,,$$
contradicting Corollary~\ref{cor:intersectionofcolemanmaps}(b).
\end{proof}
\begin{proposition}
\label{prop:specialelementstransversaltopmsubmodules}
There is an $\uI\in \fI_p$ such that 
$$\loc\left(H^1_{\Iw,S}(F,T)_+\right)\cap H^1_{\uI}(F_p,T)_+=\{0\}.$$
\end{proposition}
\begin{proof}
This is immediate from Lemma~\ref{lemma:extendto2g} by setting $W=\loc\left(\HIw(F,T)_+\right)$, since we assumed the weak Leopoldt conjecture.\end{proof}


Let $T^\dagger=T^*\otimes\mu_{p^\infty}$ denote the Cartier dual of $T$. The standard Selmer group $\Sel(T^\dagger/F(\mu_{p^\infty}))$ is defined to be
\[
\Sel(T^\dagger/F(\mu_{p^\infty})):=\ker\left(H^1(F(\mu_{p^\infty}),T^\dagger)\rightarrow\bigoplus_v\frac{H^1(F(\mu_{p^\infty})_v,T^\dagger)}{H^1_f(F(\mu_{p^\infty})_v,T^\dagger)}\right).
\]
We shall modify the conditions at primes above $p$ using our Coleman maps. 

Fix  $\uI\in\fI_p$. By local Tate duality, there is a pairing
\begin{equation}\label{eq:Tpairing}
\HIw(F_p,T)\times H^1(F(\mu_{p^\infty})_p,T^\dagger)\rightarrow\Qp/\Zp
\end{equation}
where $H^1(F(\mu_{p^\infty})_p,T^\dagger)$ denotes $
\bigoplus_{\fp|p}H^1(F(\mu_{p^\infty})_\fp,T^\dagger).$

Define $H^1_{\uI}(F(\mu_{p^\infty})_p,T^\dagger)$ to be the orthogonal complement of  $H^1_{\uI}(F_p,T)$ under the pairing \eqref{eq:Tpairing}.

\begin{defn}
\label{def:modifiedselmergroup}
We define the $\uI$-Selmer group $\Sel_{\uI}(T^\dagger/F(\mu_{p^\infty}))$ to be
\[
\ker\left(  H^1(F(\mu_{p^\infty}),T^\dagger) \longrightarrow
\bigoplus_{v\nmid p}\frac{H^1(F(\mu_{p^\infty})_v,T^\dagger)}{H^1_f(F(\mu_{p^\infty})_v,T^\dagger)}\oplus\frac{H^1(F(\mu_{p^\infty})_p,T^\dagger)}{H^1_{\uI}(F(\mu_{p^\infty})_p,T^\dagger)}\right).
\]
\end{defn}
\begin{remark}
Let $A/\QQ$ be an abelian variety of dimension $g$ and $A^\vee$ denote its dual abelian variety. Throughout this remark we set $T=T_p(A)$, the $p$-adic Tate module of $A$. In this case, we have for the local conditions that determine the standard Selmer group that
$$H^1_f(\Qp(\mu_{p^\infty}),T^\dagger)=A^\vee\left(\Qp(\mu_{p^\infty})\right)\,.$$
When $A$ has good ordinary reduction at $p$, the $\Lambda$-module $A^\vee\left(\Qp(\mu_{p^\infty})\right)$ has corank $g$ by the main result of \cite{schneider87} and $\Sel(A^\vee/\QQ(\mu_{p^\infty}))$ is predicted to be $\Lambda$-cotorsion. In the supersingular case, however, the module $A^\vee\left(\Qp(\mu_{p^\infty})\right)$ has $\Lambda$-corank $2g$, thus $\Sel(A^\vee/\QQ(\mu_{p^\infty}))$ has corank at least $g$. In the definition above we replace the local conditions $A^\vee\left(\Qp(\mu_{p^\infty})\right)$ that appear in the definition of the standard Selmer group by a corank-$g$ submodule and conjecture that the resulting Selmer groups are $\Lambda$-cotorsion.
\end{remark}

\begin{proposition}\label{prop:Ttorsion}
For $\uI\in\fI_p$ verifying the conclusion of Proposition~\ref{prop:specialelementstransversaltopmsubmodules} the $\Lambda_+$-module $\Sel_{\uI}(T^\dagger/F(\mu_{p^\infty}))_+$ is cotorsion.
\end{proposition}
\begin{proof}
It follows from our choice of $\uI$ and Poitou-Tate global duality that we have an exact sequence
\begin{equation}\label{PTsequence}0\rightarrow H^1_{\Iw,S}(F,T)_+\rightarrow \frac{ \HIw(F_p,T)_+}{H^1_{\uI}(F_p,T)_+}\rightarrow \Sel_{\uI}(T^\dagger/F(\mu_{p^\infty}))_+^{\vee}\rightarrow H^2_{\Iw,S}(F,T)_+\rightarrow 0\end{equation}
 The $\Lambda_+$-module $H^2_{\Iw,S}(F,T)_+$ is torsion whereas the $\Lambda_+$-module $H^1_{\Iw,S}(F,T)_+$ has $\Lambda_+$-rank $g_-$ by the weak Leopoldt conjecture that we assume. Proposition follows by counting $\Lambda_+$-ranks in the sequence (\ref{PTsequence}).
\end{proof}
\begin{remark}
We expect that the $\Lambda_+$-module $\Sel_{\uI}(T^\dagger/F(\mu_{p^\infty}))_+$ is cotorsion for \emph{every} $\uI\in\fI_p$. However, we are able to verify this guess (assuming weak Leopoldt conjecture for $T$) for only one $\uI$. This is fortunately sufficient for our purposes.
\end{remark}

The following statement will be referred as the \emph{$\uI$-main conjecture.} We shall verify that its truth for a single $\uI \in \frak{J}_p$ is equivalent to the $\eta$-isotypic part of Perrin-Riou's main Conjecture \ref{TCP}.
{
\begin{conj}\label{myTCP}
Let $\uI\in\fI_p$ and $\eta$ an even Dirichlet character of conductor $p$. Then
\[
\Char_{\Zp[[X]]}\Sel_{\uI}(T^\dagger/F(\mu_{p^\infty}))^{\vee,\eta}=\left(\frac{L_{\uI}(\mathcal{M}^*(1))^\eta}{X^{n(\uI,\eta)}}\right)\Zp[[X]], 
\]where $n(\uI,\eta)$ is the integer as given by Lemma~\ref{lem:Timage}. 
\end{conj}
}

\begin{theorem}
\label{thm:TmyCPeqivtoCP}
Assume the truth of the Explicit Reciprocity Conjecture~\ref{conj:Tspecialelement} for the module of Kolyvagin-determinants. For every even Dirichlet character $\eta$ of $\Delta$, the $\eta$-part of Conjecture~\ref{TCP} is equivalent to Conjecture~\ref{myTCP} for every $\uI \in \frak{I}_p$ verifying the conclusion of Proposition~\ref{prop:specialelementstransversaltopmsubmodules}.  
\end{theorem}
\begin{proof}
Recall the Poitou-Tate exact sequence (\ref{PTsequence}):
$$0\rightarrow H^1_{\Iw,S}(F,T)^\eta\rightarrow\frac{\HIw(F_p,T)^\eta}{H^1_{\uI}(F_p,T)^\eta}\rightarrow\Sel_{\uI}(T^\dagger/F(\mu_{p^\infty}))^{\vee,\eta}\rightarrow H^2_{\Iw,S}(F,T)^\eta\rightarrow0\,.$$

Note that the left-most injection follows from the choice of $\uI$. The second term in \eqref{PTsequence} is isomorphic to $\image\left(\col_T^{\uI}\right)^\eta$, which is described by Proposition~\ref{prop:iageofcolemanis}.

Let $\cc=\cc_1\wedge\cdots\wedge\cc_{g_-}$ be the element given by Conjecture~\ref{conj:Tspecialelement}. The exact sequence \eqref{PTsequence} then yields the following exact sequence:
\begin{align*}
0\longrightarrow H^1_{\Iw,S}(F,T)^\eta/\left(\textup{span}_{\Lambda}\{c_i\}_{i=1}^{g_-}\right)^\eta\longrightarrow{\image\left(\col_T^{\uI}\right)^\eta}\Big{/}{\left({\textup{span}_{\Lambda}\,\left\{\col_{\uI}(\cc_i)\right\}_{i=1}^{g_-}}\right)^\eta}\\
\longrightarrow\Sel_{\uI}(T^\dagger/F(\mu_{p^\infty}))^{\vee,\eta}\longrightarrow H^2_{\Iw,S}(F,T)^\eta\longrightarrow0.
\end{align*}
We therefore conclude
\begin{align*}
\det \left(H^1_{\Iw,S}(F,T)^\eta/\left(\textup{span}_{\Lambda}\{c_i\}_{i=1}^{g_-}\right)^\eta\right)\otimes\det\left(\Sel_{\uI}(T^\dagger/F(\mu_{p^\infty}))^{\vee,\eta}\right)
=\\
\det\left({\image\left(\col_T^{\uI}\right)^\eta}\Big{/}{\left(\textup{span}_{\Lambda}\,\left\{\col_{\uI}(\cc_i)\right\}_{i=1}^{g_-}\right)^\eta}\right)\otimes\det\left(H^2_{\Iw,S}(F,T)^\eta\right),
\end{align*}
which can be rewritten as
\[
e_\eta\cdot\ff_\cc^{-1}\det\left(\Sel_{\uI}(T^\dagger/F(\mu_{p^\infty}))^\vee\right)
=e_\eta\cdot\det\left({\image\left(\col_T^{\uI}\right)}\Big{/}{\textup{span}_{\Lambda}\,\left\{\col_{\uI}(\cc_i)\right\}_{i=1}^{g_-}}\right)\ff_2^{-1}.
\]
By Proposition~\ref{propn:TCP}, it follows that Conjecture~\ref{TCP} is equivalent to
\[
\det\left(\Sel_{\uI}(T^\dagger/F(\mu_{p^\infty}))^{\vee,\eta}\right)
=\det\left({\image\left(\col_T^{\uI}\right)^\eta}\Big{/}\left({\textup{span}_{\Lambda}\,\left\{\col_{\uI}(\cc_i)\right\}_{i=1}^{g_-}}\right)^\eta\right).
\]
Hence we are done by Lemma~\ref{lem:Tdet}.
\end{proof}



\begin{theorem}
\label{thm:PRmainconj}
Suppose that the representation $T$ verifies the hypotheses \textup{\textbf{(H1)-(H4)}} of \cite[\S3.5]{mr02} and assume the truth of Conjecture~\ref{conj:Tspecialelement}. Then following containment 
\[L_p(\mathcal{M}^*(1))\Lambda_+\subset \Iarith(T)_+.\]
in the statement of Perrin-Riou's Main Conjecture~\ref{TCP} holds true.
\end{theorem}
\begin{proof}
Choose $\uI\in\fI_p$ verifying the conclusion of Proposition~\ref{prop:specialelementstransversaltopmsubmodules}. Let $\eta$ be an even character of $\Delta$. 
In what follows we will freely borrow notation and concepts from Appendix~\ref{sec:appendixKSexists}. Let $\pmb{\kappa}\in \overline{\mathbf{KS}}(\mathbb{T}\otimes\eta^{-1},\FFF_{\mathbb{L}},\mathcal{P}_X)$ be \emph{any} generator of the module $\Lambda^{(p)}$-adic Kolyvagin systems and $\kappa_1\in  H^1_{\mathcal{F}_{\mathbb{L}}}(F,\mathbb{T}\otimes\eta^{-1})$ denote its initial term. Recall that $\Lambda^{(p)}=\ZZ_p[[\Gamma]]$ and $\Gamma$ is the Galois group of the cyclotomic $\ZZ_p$-tower, so that $\Lambda^{(p)}\cong \ZZ_p[[X]]$. Poitou-Tate global duality yields an exact sequence
\begin{align*}
0\longrightarrow H^1_{\mathcal{F}_{\mathbb{L}}}(F,\mathbb{T}\otimes\eta^{-1})/&\Lambda\cdot\kappa_1\stackrel{\loc}{\longrightarrow} \frac{H^1_{\mathcal{F}_{\mathbb{L}}}(F_p,\mathbb{T}\otimes\eta^{-1})}{H^1_{\uI}(F_p,{T})^\eta+\Lambda^{(p)}\cdot\loc({\kappa_1)}}\\
&\longrightarrow \Sel_{\uI}(T^\dagger/F(\mu_{p^\infty}))^{\vee,\eta}\longrightarrow H^1_{\mathcal{F}_{\mathbb{L}}^*}(F,\mathbb{T}^{\dagger}\otimes\eta)^\vee\longrightarrow 0
\end{align*}
We then have{
\begin{align*}\textup{char}\left(\Sel_{\uI}(T^\dagger/F(\mu_{p^\infty}))^{\vee,\eta}\right)&=\textup{char}\left(\frac{H^1_{\mathcal{F}_{\mathbb{L}}}(F_p,\mathbb{T}\otimes\eta^{-1})}{H^1_{\uI}(F_p,{T})^\eta+\Lambda^{(p)}\cdot\loc({\kappa_1)}}\right)\\
&=\frac{\col_T^{\uI,\eta}\left(\loc(\kappa_1)\right)}{X^{n(\uI,\eta)}}\cdot \Lambda^{(p)}\\
&\supset \frac{\det\left(\col_T^{\uI,\eta}(\cc_i)\right)}{X^{n(\uI,\eta)}}\cdot\Lambda^{(p)} \\
&=\left(\frac{L_{\uI}(\mathcal{M}^*(1))^\eta}{X^{n(\uI,\eta))}}\right)\cdot\Lambda^{(p)}\end{align*} }
where 
\begin{itemize}
\item the first equality follows from Theorem~\ref{thm:appendixBmain}(iii),
\item the second using the fact that $\col^{\uI,\eta}_T$ is injective (by very definitions) on the quotient ${H^1_{\mathcal{F}_{\mathbb{L}}}(F_p,\mathbb{T}\otimes\eta^{-1})}/{H^1_{\uI}(F_p,{T})^\eta}$, has pseudo-null cokernel by Propsition~\ref{prop:iageofcolemanis}, and by fixing a generator of $\mathbb{L}$, 
\item the third using the fact that $\cc \in \frak{K}(T)$ and the commutativity of the diagram~(\ref{eqn:diagramappendixC}),
\item and finally the last by Lemma~\ref{lem:Tdet} and the fact that we have chosen of our Coleman maps relative to a strongly admissible basis.
\end{itemize}
This verifies the containment 
\begin{equation}\label{eqn:containmentsignedMC}\left(L_{\uI}(\mathcal{M}^*(1))^\eta\right)\Zp[[X]]\subset \Char_{\Zp[[X]]}\Sel_{\uI}(T^\dagger/F(\mu_{p^\infty}))^{\vee,\eta}\end{equation}
in the statement of Conjecture~\ref{myTCP}. We conclude the proof of the theorem on using (\ref{eqn:containmentsignedMC}) together with the proof of Theorem~\ref{thm:TmyCPeqivtoCP}. 
\end{proof}
\begin{remark}
\label{rem:hypoholdtrue}
See \cite{buyukboduklei1} for an example where we deduce an explicit version of Theorem~\ref{thm:PRmainconj}. In loc.cit., we study more closely the motive attached to the Hecke character associated to a CM abelian variety that has supersingular reduction at all primes above $p$. In this particular case, the hypotheses  \textup{\textbf{(H1)-(H4)}} of \cite[\S3.5]{mr02}, \textup{(H.F.-L.)}, \textup{(H.S.)}, \textup{(H.P.)} and \textup{(H.nA)} hold true.
\end{remark}

\appendix

\section{An alternative approach using Wach modules}

In \cite{leiloefflerzerbes10} and \cite{leiloefflerzerbes11}, we showed that the theory of Wach modules can be used to study the Iwasawa theory of $p$-adic representations. The key is to find an explicit basis for the Wach module. In this appendix, we show that the construction of the logarithmic matrix $M_T$ in \S\ref{S:matrix} can be modified to construct an explicit basis for the Wach module $\NN(T)$ of $T$. Here $T$ is as defined in \S\ref{S:matrix}, satisfying (H.F.-L.) and (H.S.).

Let $\AQp=\calO_K[[\pi]]$, which is equipped with the usual semi-linear actions by $\Gamma$ and $\vp$ (see for example \cite{berger04}). We write $q=\vp(\pi)/\pi$. 

\begin{defn}
A Wach module with weights in $[a;b]$ is a finitely generated free $\AQp$-module $M$ such that
\begin{itemize}
\item[(1)] It is equipped with a semi-linear action by $\Gamma$ that is trivial modulo $\pi$;
\item[(2)] There is a semi-linear map $\vp:M[\pi^{-1}]\rightarrow M[\vp(\pi)^{-1}]$ such that $\vp(\pi^b M)\subset \pi^b M$ and $q^{b-a}\vp(\pi^bM)\subset \pi^bM$;
\item[(3)] The actions of  $\Gamma$ and $\vp$ commute.
\end{itemize}
\end{defn}

A Wach module $N$ is equipped with a filtration
\[
\Fil^iN=\{x\in N:\vp(x)\in q^iN\}.
\]

Let $v_1,\ldots,v_d$ be $\calO_K$-basis of $\DD_K(T)$ such that $v_1,\ldots,v_{d_0}$ generate $\Fil^0\DD_K(T)$. Let $C_\vp$ be the matrix of $\vp$ with respect to this basis. As in \S\ref{S:matrix},{
\[
C_\vp=C\left(
\begin{array}{c|c}
I_{r_0}&0\\ \hline
0&\frac{1}{p} I_{r-r_0}
\end{array}
\right)
\]}
for some $C\in\GL_d(\calO_K)$.
\begin{defn}\label{defn:Pn}
For $n\ge1$, we define{
\[
P_n=C
\left(
\begin{array}{c|c}
I_{r_0}&0\\ \hline
0&\frac{1}{\vp^{n-1}(q)} I_{r-r_0}
\end{array}
\right)
\quad\text{and}\quad
M'_n=\left(C_\vp\right)^{n}P_{n}^{-1}\cdots P_{1}^{-1}.
\]}
\end{defn}

\begin{proposition}
The sequence of matrices $\{M'_n\}_{n\ge1}$ converges entry-wise with respect to the sup-norm topology on $\Brig$. If $M_T'$ denotes the limit of the sequence, each entry of $M_T'$ are $o(\log)$. Moreover, $\det(M'_T)$ is, up to a constant in $\calO_K^\times$, equal to $\left(\frac{\log(1+\pi)}{\pi}\right)^g$.
\end{proposition}
\begin{proof}
The proof is the same as that for Proposition~\ref{prop:matrix}.
\end{proof}

\begin{defn}
For each $\gamma\in\Gamma$, define a matrix $G_\gamma=\left(M_T'\right)^{-1}\cdot\gamma\left(M_T'\right)$.
\end{defn}

We shall show that $G_\gamma$ is a matrix defined over $\AQp$. Let us first prove the following lemma.

\begin{lemma} \label{lem:matrixdefn}Let $M_{r\times r}(\AQp)$ be the set of $r\times r$ matrices that are defined over $\AQp$.
\begin{itemize}
\item[(a)] $P_1\cdot \gamma\left(P_1^{-1}\right)\in I+\pi M_{r\times r}(\AQp)$;
\item[(b)] If $M\in I+\pi M_{r\times r}(\AQp)$, then $P_1\cdot \vp(M)\cdot\gamma(P_1^{-1})\in I+\pi M_{r\times r}(\AQp)$.
\end{itemize}
\end{lemma}
\begin{proof}
 For (a), we have {$P_1\cdot\gamma(P_1^{-1})=C\left(
\begin{array}{c|c}
I_{r_0}&0\\ \hline
0&\frac{\gamma\cdot q}{q} I_{r-r_0}
\end{array}
\right)C^{-1}$} and $\frac{\gamma\cdot q}{q}\in 1+\pi\AQp$, hence the result.

Let $M=I+\pi N$, then
\[
P_1\cdot \vp(M)\cdot\gamma(P_1^{-1})=P_1\gamma\left(P_1^{-1}\right)+\pi\left(qP_1\cdot \vp(N)\cdot\gamma\left(P_1^{-1}\right)\right)
\]
since $\vp(\pi)=\pi q$. Both $qP_1$ and $P_1^{-1}$ are defined over $\AQp$, so (b) follows from (a).
\end{proof}

\begin{proposition}\label{propn:actiongamma}
For all $\gamma$, the matrix $G_\gamma$ is an element of $I+\pi M_{r\times r}(\AQp)$.
\end{proposition}
\begin{proof}
Since $G_\gamma=\lim_{n\rightarrow\infty}\left(M'_n\right)^{-1}\cdot\gamma\left(M_n'\right)$, it is enough to show that $\left(M'_n\right)^{-1}\cdot\gamma\left(M'_n\right)$ is in $I+\pi M_{r\times r}(\AQp)$ for all $n$. Let us show this by induction.

We have for all $n$
\begin{equation}\label{eq:N}
\left(M'_n\right)^{-1}\cdot\gamma\left(M'_n\right)=P_1\cdots P_{n}\gamma(P_{n}^{-1})\cdots\gamma(P_1^{-1}).
\end{equation}
Hence, the claim for $n=1$ is Lemma~\ref{lem:matrixdefn}(a).

By definition, $P_n=\vp^{n-1}(P_1)$, so we have for $n\ge2$
\[
\left(M'_n\right)^{-1}\cdot\gamma\left(M'_n\right)=P_1\cdot\vp\left(\left(M'_n\right)^{-1}\cdot\gamma\left(M'_n\right)\right)\cdot\gamma(P_1^{-1}).
\]
Hence, the inductive step is simply Lemma~\ref{lem:matrixdefn}(b).
\end{proof}

\begin{lemma}\label{lem:commute}
For all $\gamma$, we have the matrix identity
\[
P_1\cdot\vp(G_\gamma)=G_\gamma\cdot\gamma(P_1).
\]
\end{lemma}
\begin{proof}
By \eqref{eq:N} and the fact that $P_n=\vp^{n-1}(P_1)$, we have
\[
P_1\cdot\vp\left(\left(M'_n\right)^{-1}\cdot\gamma\left(M'_n\right)\right)=P_1\cdots P_{n+1}\gamma(P_{n+1}^{-1}\cdots P_2^{-1})
\]
and
\[
\left(\left(M'_n\right)^{-1}\cdot\gamma\left(M'_n\right)\right)\cdot\gamma(P_1)=P_1\cdots P_n\gamma(P_n^{-1}\cdots P_2^{-1}).
\]
In other words,
\[
P_1\cdot\vp\left(\left(M'_n\right)^{-1}\cdot\gamma\left(M'_n\right)\right)=\left(\left(M'_{n+1}\right)^{-1}\cdot\gamma\left(M'_{n+1}\right)\right)\cdot\gamma(P_1)
\]
Hence the result follows on taking $n\rightarrow\infty$.
 \end{proof}

\begin{defn}
We define a free $\AQp$-module $N_{C_\vp}$ of rank $r$, with basis $n_1,\ldots, n_{r}$. With respect to this basis, we equip $N_{C_\vp}$ with a semi-linear action by $\Gamma$, which is given by the matrix $G_\gamma$ (well-defined by Proposition~\ref{propn:actiongamma}) and a semi-linear map $\vp:N_{C_\vp}[\pi^{-1}]\rightarrow N_{C_\vp}[\vp(\pi)^{-1}]$, which is given by the matrix $P_1$.
\end{defn}

\begin{proposition}
The module $N_{C_\vp}$ is a Wach  module with weights in $[0;1]$.
\end{proposition}
\begin{proof}
By Proposition~\ref{propn:actiongamma}, the action of $\Gamma$ on $N_{C_\vp}$ is trivial modulo $\pi$.

Since $P_1\in1/q M_{r\times r}(\AQp)$, we have \[
\vp\left(\pi N_{C_\vp}\right)\in\pi N_{C_\vp}\quad\text{and}\quad q\vp\left(N_{C_\vp}\right)\subset \pi^bN_{C_\vp}.
\]

Finally, by Lemma~\ref{lem:commute}, the actions of $\Gamma$ and $\vp$ commute, so we are done. 
\end{proof}

\begin{theorem}
As Wach modules, $N_{C_\vp}$ is isomorphic to $\NN(T)$. Furthermore, 
\[
\begin{pmatrix}v_1&\cdots&v_{r}\end{pmatrix}M_T'=\begin{pmatrix}n_1&\cdots&n_{r}\end{pmatrix}.
\]
\end{theorem}
\begin{proof}
In order to show that $N_{C_\vp}\cong\NN(T)$, it is enough to show that
\begin{equation}\label{eq:compareWach}
\DD_K(T)\cong N_{C_\vp}\mod\pi
\end{equation}
as filtered $\vp$-module by \cite[Th\'eor\`me~III.4.4]{berger04}.

By definition $P_1\equiv C_\vp\mod \pi$, so the actions of $\vp$ agree on the two sides of \eqref{eq:compareWach}. For the filtration, we have
\[
\Fil^iN_{C_\vp}=
\begin{cases}
N_{C_\vp}&i\le-1\\
\left(\bigoplus_{1\le j\le r_0}\AQp\cdot n_j\right)\oplus\left(\bigoplus_{r_0+1\le j\le r}\AQp\cdot \pi n_j\right)&i=0\\
\left(\bigoplus_{1\le j\le r_0}\AQp\cdot \pi^in_j\right)\oplus\left(\bigoplus_{r_0+1\le j\le r}\AQp\cdot \pi^{i+1}n_j\right)&i\ge1
\end{cases}.\]
Since $\Fil^0\DD(T_p(A))$ is generated by $v_1,\ldots,v_{r_0}$, we see that the filtrations on the two sides of \eqref{eq:compareWach} as well.

By \cite[\S II.3]{berger04},
\begin{equation}\label{eq:relationWach}
\begin{pmatrix}v_1&\cdots&v_{r}\end{pmatrix}M=\begin{pmatrix}n_1&\cdots&n_{r}\end{pmatrix}.
\end{equation}
for some matrix $M\in I+\pi M_{r\times r}(\Brig)$. For any $\gamma\in\Gamma$, 
\[\begin{pmatrix}v_1&\cdots&v_{r}\end{pmatrix}\gamma(M)=\begin{pmatrix}n_1&\cdots&n_{r}\end{pmatrix}G_\gamma.\]
Therefore, $G_\gamma=M\cdot\gamma(M^{-1})=M_T'\cdot\gamma\left(M_T'\right)^{-1}$. But $M_T'\in I+\pi M_{r\times r}(\Brig)$ also. Hence,
\[
M\cdot\left(M_T'\right)^{-1}\in\left(I+\pi M_{r\times r}(\Brig)\right)^\Gamma.
\]
This implies that $M=M_T'$ as required.
\end{proof}

We now use the theory of Wach modules to prove an integrality result that is used in the main part of the article. Recall from \cite[\S3.1]{leiloefflerzerbes10} and \cite[\S3.1]{leiloefflerzerbes11} that for any $x\in\NN(T)^{\psi=0}$, we have $(1-\vp)x\in (\vp^*\NN(T))^{\psi=0}\subset\Brig\otimes\DD_K(T)$. Furthermore, we have a $\calO_K\otimes\Lambda$-basis for  $(\vp^*\NN(T))^{\psi=0}$ of  the form $(1+\pi)\vp(n_1),\ldots,(1+\pi)\vp(n_r)$

\begin{lemma}\label{lem:integralityPR}
Let $x\in\NN(T)^{\psi=1}$, then $(1\otimes\vp^{-n-1})\circ(1-\vp)x$ is congruent to an element of $(\AQp)^{\psi=0}\otimes\DD_K(T)$ modulo $\vp^{n+1}(\pi)\Brig\otimes\DD_K(T)$.
\end{lemma}
\begin{proof}
By \cite[Lemma~3.3]{leiloefflerzerbes10}, there exists $x_1,\ldots,x_d\in(\AQp)^{\psi=0}$ such that
\[
(1-\vp)x=\sum_{i=1}^rx_i(1+\pi)\vp(n_i)=\begin{pmatrix}
v_1&\hdots&v_r
\end{pmatrix}\cdot C_\vp\cdot (1+\pi)\vp(M)\cdot\begin{pmatrix}
x_1\\ \vdots \\ x_d
\end{pmatrix} .
\]Note that we have abused notation to write $v_i\cdot (\star)$ for $(\star)\otimes v_i\in\Brig\otimes\DD_K(T)$.
Thus, on applying $(1\otimes\vp^{-n-1})$, we have
\[
(1\otimes\vp^{-n-1})\circ(1-\vp)x=\begin{pmatrix}
v_1&\hdots&v_r
\end{pmatrix}\cdot C_\vp^{-n}\cdot \vp(M)\cdot\begin{pmatrix}
x_1\\ \vdots \\ x_d
\end{pmatrix} .
\]
Therefore, it is enough to show that $C_\vp^{-n}\cdot \vp(M)$ is congruent to some element in $\AQp$ modulo $\vp^{n+1}(\pi)\Brig$.

If we apply $\vp$ to the equation \eqref{eq:relationWach}, we have the relation
\[
M=C_\vp\cdot\vp(M)\cdot P^{-1}.
\]
Since $M\equiv I\mod \pi$, we have $M\equiv C_\vp\cdot P^{-1}$ modulo $\pi$. On iterating, we have
$$M\equiv C_\vp^{n}\cdot \vp^{n-1}(P^{-1})\cdots P^{-1}\mod\vp^n(\pi),$$
which implies that
$$\vp(M)\equiv C_\vp^{n}\cdot \vp^{n}(P^{-1})\cdots \vp(P^{-1})\mod\vp^{n+1}(\pi).$$
Recall that $P^{-1}$ is defined over $\AQp$, hence we are done.
\end{proof}


\section{Linear Algebra: Proof of Proposition~\ref{prop:admissiblebasisexists}}
\label{appendix:linearalgebra}
The goal of this appendix is to provide a proof of Proposition~\ref{prop:admissiblebasisexists}. 
\begin{lemma}
\label{lemma:linearalgebra1}
Let $W$ be a free $\ZZ_p$-module of rank $\frak{d}$ and let $W^\prime$ be a free, rank $\frak{d}-1$ direct summand of $W$. Then the collection  $\{W^\prime+\ZZ_p\cdot v:  v\in W\}$ of submodules of $W$ is totally ordered (with respect to inclusion).
\end{lemma}
\begin{proof}
This follows from the fact that the quotient $W/W^\prime$ is a free $\ZZ_p$-module of rank one.
\end{proof}
\begin{lemma}
\label{lemma:codim1issmall}
Let $W$ be as in the previous lemma. Let $\frak{D}$ be a finite collection of rank $\frak{d}-1$ direct summands of $W$ and let $W_0=\cup_{\frak{D}}\,W^\prime$ be their union. For any $k\in \ZZ^+$ we have,
$$p^kW\cup W_0\neq W.$$
\begin{proof}
Choose any element $w=w_0 \in W-W_0$ (such an element clearly exists). If $w_0\not\in p^kW$, we are done, otherwise write $w_0=p^k w_1$. Observe that $w_1\not\in W_0$ (as otherwise, $w_0$ would be an element of $W_0$ as well). Now if $w_1\not\in p^kW$, we are done again. Otherwise we may continue with this process, which eventually has to terminate.
\end{proof}
\end{lemma}
\begin{lemma}
\label{lemma:slopeavoidance}
For $\left(\begin{array}{cc}a & b\\ c& d \end{array} \right)\in \GL_2(\ZZ_p)$, the set 
$\left\{\frac{ax+by}{cx+dy}: x,y \in \ZZ_p^\times\right\}$
has infinite cardinality.
\end{lemma}

\begin{proof}
Since  $\left(\begin{array}{cc}a & b\\ c& d \end{array} \right)\in \GL_2(\ZZ_p)$, either $c\neq0$ or $d\neq0$; say the first holds true.  Note that $\displaystyle{\frac{ax+by}{cx+dy}=\frac{a}{c}-\frac{(ad-bc)/{c}}{cx+dy}}$, and $ad-bc\neq 0$  and that $cx+dy$ takes on infinitely many values as $x,y\in \ZZ_p^\times$ vary.
\end{proof}
\begin{lemma}
\label{lemma:commonv}
Let $W$, $\frak{D}$ and $W_0$ be as in Lemma~\ref{lemma:codim1issmall}. Let $W_1,W_2 \in \frak{D}$ and suppose $v_1,v_2 \in W-W_0$ verify 
$$W_1\oplus\ZZ_p\cdot v_1=W=W_2\oplus\ZZ_p\cdot v_2\,.$$
There one can choose $\alpha,\beta \in \ZZ_p$ so that
\begin{enumerate}
\item[(a)] $v=\alpha v_1+ \beta v_2 \in W-W_0$,
\item[(b)] $W_1\oplus\ZZ_p\cdot v=W_2\oplus\ZZ_p\cdot v=W$.
\end{enumerate}
\end{lemma}
\begin{proof}
Fix a basis $\frak{B}_1$ of $W_1$ and $\frak{B}_2$ of $W_2$. Let $x_1$ be the $v_2$-coordinate of $v_1$ with respect to the basis $\frak{B}_2\cup \{v_2\}$ and $x_2$ be the $v_1$-coordinate of $v_2$ with respect to the basis $\frak{B}_1 \cup \{v_1\}$. We may assume without loss of generality that $v_p(x_i)>0$ for $i=1,2$, as otherwise, say in case $v_p(x_1)=0$, it would follow that $\textup{span}\left(\frak{B}_2,v_1\right)=\textup{span}\left(\frak{B}_2,x_1\cdot v_2\right)=W$ and thus the choice $\alpha=1$ and $\beta=0$ (thus $v=v_1$) would work. Let $X=\left(\begin{array}{cc}x_1 & 1\\ 1& x_2 \end{array} \right)$ and let $Y=\left(\begin{array}{cc}a & b\\ c& d \end{array} \right)\in\GL_2(\ZZ_p)$ be such that $YX=1$ (such $Y$ exists since $\det(X)\in \ZZ_p^\times$ thanks to our running hypothesis).

Consider $W_0\, \cap\, \textup{span}\left(v_1,v_2\right)$. Since $v_1 \not\in W_0$, it follows that this intersection is a finite union of $\ZZ_p$-lines, say spanned by $\{\alpha_iv_1+\beta_iv_2\}_{i=1}^d$ (with $\alpha_i,\beta_i \in \ZZ_p$). Let $\frak{X}=\{\alpha_i/\beta_i : \beta_i\neq 0\}$, note that it is a finite subset of $\QQ_p$. Use Lemma~\ref{lemma:slopeavoidance} to choose $x,y \in \ZZ_p^\times$ such that $\frac{ax+by}{cx+dy} \not\in \frak{X}$. Set $\alpha=ax+by$ and $\beta=cx+dy$. Note that we have by definitions
$$Y\left[\begin{array}{c}x\\y \end{array}\right]=\left[\begin{array}{c}\alpha\\ \beta \end{array}\right],$$
or equivalently that
\begin{equation} \label{equation:thekeyequation} \left(\begin{array}{cc}x_1 & 1\\ 1& x_2 \end{array} \right)\left[\begin{array}{c}\alpha\\ \beta \end{array}\right]=X\left[\begin{array}{c}\alpha\\ \beta \end{array}\right]=\left[\begin{array}{c}x\\y \end{array}\right]\,.
\end{equation}

Observe that $v:=\alpha v_1 +\beta v_2 \not\in W_0$ (as $\alpha/\beta\not\in\frak{X}$), so $v$ satisfies (a). Furthermore, 
$$ v=\alpha v_1 +\beta v_2  \equiv (\alpha x_1 + \beta)\cdot v_2=x\cdot v_2 \mod W_2$$
and
$$ v \equiv (\alpha  + \beta x_2)\cdot v_1=y\cdot v_1 \mod W_1$$
We therefore conclude (using the fact $x,y \in \ZZ_p^\times$) that 
$$\textup{span}\left(W_1,v\right)=\textup{span}\left(W_1,y\cdot v_1\right)=\textup{span}\left(W_1,v_1\right)=W\,,$$
and 
$$\textup{span}\left(W_2,v\right)=\textup{span}\left(W_2,x\cdot v_2\right)=\textup{span}\left(W_1,v_2\right)=W\,,$$
which proves that $v$ verifies (b) as well.
\end{proof}
\begin{lemma}
\label{lemma:linearalgebra2}
Let $W$ be as in the previous lemma and let $\{w_1,\dots,w_\frak{d}\}$ be a given basis of $W$. For any non-negative integer $k$, one can find elements $\{w_{\frak{d}+1},\dots, w_{\frak{d}+k}\}\subset W$ so that for any $I \subset \{1,\dots,\frak{d}+k\}$ of size $\frak{d}$, the set $\{w_j\}_{j\in I}$ spans $W$.
\end{lemma}
\begin{proof}
We prove the lemma by induction on $k$. When $k=0$, the assertion is clear and suppose that $k\geq 1$ we have found a set  $\{w_{\frak{d}+1},\dots, w_{\frak{d}+k-1}\}$. Let $\frak{S}$ denote the collection of subsets of $1,\dots,\frak{d}+k-1$ of size $\frak{d}-1$ and let $\frak{D}=\{\textup{span}\left(\{w_i\}_{i\in S}\right): S\in\frak{S}\}$ be a set of free, rank $\frak{d}-1$ direct summands of $W$. Set $W_0=\cup_{\frak{D}}\,W^\prime$, observe that $W_0$ is a proper subset of $W$. For any $w \in W- W_0$ and $S\in \frak{S}$, the submodule $\textup{span}\left(\{w\}\cup \{w_i\}_{i \in S}\right)$ of $W$ is of finite index. Fix $S \in \frak{S}$ and define $W_S:=\textup{span}\left(w_i: i \in S\right).$ 

We first prove that there is an element $v_S\in W-W_0$ such that 
\begin{equation}
\label{eqn:vS} W_S+\ZZ_p\cdot v_S=W.
\end{equation}
Indeed, pick any $w\in W-W_0$. If $W_S+\ZZ_p\cdot w=W$, we are done. Otherwise we may use Lemma~\ref{lemma:codim1issmall} to choose  $w_1\in W-(W_S+\ZZ_p\cdot v \cup W_0)$, for which we have 
$$ W_S+\ZZ_p\cdot w_1\supsetneq W_S+\ZZ_p\cdot w\,.$$
This process has to terminate and when it does, we have found the desired $v_S$ satisfying (\ref{eqn:vS}).

Using Lemma~\ref{lemma:commonv} iteratively, one obtains an element $v \in W-W_0$ such that 
$$W_S+\ZZ_p\cdot v= W$$
for every $S \in \frak{S}$. We set $w_{\frak{d}+k}:=v$.
\end{proof}

\begin{proof}[Proof of Proposition~\ref{prop:admissiblebasisexists}]
Let $\frak{B}=\{v_1,\cdots,v_{g_-},w_{g_-+1},\dots,w_{g}\}$ be any $\ZZ_p$-basis of $\DD_p(T)$ such that $\{v_1,\cdots,v_{g_-}\}$ forms a basis of $\Fil^0\DD_p(T)$. Form the dual basis $$\frak{B}^\prime=\{v_1^\prime,\cdots,v_{g_-}^\prime,w_{g_-+1}^\prime,\cdots,w_{g}^\prime\}\subset\DD_p(T^*(1)).$$ 
Consider the free $\ZZ_p$-module  $W:=\DD_p(T^*(1))/\Fil^0\DD_p(T^*(1))$ of rank $g_-$ and for an element $v\in \DD_p(T^*(1))$, let $\bar{v}$ denote its image in $W$. It is easy to see that  $\{\bar{v}_1^\prime,\cdots,\bar{v}_{g_{-}}^\prime\}$ forms a basis of $W$. Use Lemma~\ref{lemma:linearalgebra2} (with $\frak{d}=g_-$ and $k=g_+$) to obtain a set $\{\bar{v}_1^\prime,\cdots,\bar{v}_{g}^\prime\}$ such that for any $\uI \in \fI_p$, 
$$\textup{span}\left(\bar{v}_i^\prime: i\in \uI\right)=W.$$
One can lift the set $\{\bar{v}_1^\prime,\cdots,\bar{v}_{g}^\prime\}$ to a basis $\frak{B}_{\textup{ad}}^\prime=\{{v}_1^\prime,\cdots,{v}_{g}^\prime\}$  of $\DD_p(T^*(1))$ and the basis $\frak{B}_{\textup{ad}}$ dual to $\frak{B}_{\textup{ad}}^\prime$ gives us an admissible basis of $\DD_p(T)$ completes the first part of the proof.

The proof of that a strongly admissible basis exists is similar and we only provide a sketch of its proof after inverting $p$. The technical details to conclude integral version of this result are identical to the arguments above we have assembled in the course of deducing the first part regarding admissibility. To ease notation, let $\mathcal{V}=\DD_p(T^*(1))\otimes\QQ_p$ and $\mathcal{W}=\Fil^0\DD_p(T^*(1))\otimes\QQ_p$. Set also $\mathcal{T}=(1-\vp)^{-1}(p\vp-1)$ and $\mathcal{W}^\prime=\mathcal{T}^{-1}(\mathcal{W})\otimes\QQ_p$\,; note that $\mathcal{T}$ is invertible thanks to our running assumptions.  Set $r=\dim \mathcal{W}= \mathcal{W}^\prime$ and $r+s=\dim\mathcal{V}$. We choose a basis $\{v^\prime_i\}$ inductively as follows:
\begin{itemize}
\item Choose any $v_1 \notin \mathcal{W} \cup \mathcal{W}^\prime$. 
\item For $k\leq s-1$, if we have chosen $v_1^\prime,\cdots,v_k^\prime$, choose $v_{k+1}^\prime \in \mathcal{V}$ as any vector so that
$$v_{k+1}^\prime \notin \left(\textup{span}(v_i^\prime: 1\leq i\leq k)+\mathcal{W}\right)\cup  \left(\textup{span}(v_i^\prime: 1\leq i\leq k)+\mathcal{W}^\prime\right)\,.$$
Note that we can do this as on the we have a union of two hyperplanes of dimension $k+r<r+s$.
\item For any $0\leq k<s$, suppose we have chosen $\frak{B}_k=\{v_1^\prime,\cdots,v_{s+k}^\prime\}$ in a way that 
$$\textup{span}\left(v_{i_j}^{\prime}: i_{j} \in I\right) \cap \left(\mathcal{W}\cup \mathcal{W}^\prime\right)=0$$
for every subset $I\subset \{1,\cdots,s+k\}$ of size $s$. (The first two steps will get us to this step with $k=0$.) 

Let $I^{(s-1)}$ denote all subsets of $I\subset \{1,\cdots,s+k\}$ of size $s-1$ and let 
$$V^{(s-1)}=\bigcup_{J\in I^{(s-1)}}\textup{span}\left(v_{i_j}: i_j \in J \right)\,.$$
This is a finite union of hyperplanes of dimension $s-1$. Now choose $v_{s+k+1}^\prime \in \mathcal{V}$ to be any element verifying
$$v_{s+k+1}^\prime \notin \left(\mathcal{W}+V^{(s-1)}\right)\cup \left(\mathcal{W}^\prime+V^{(s-1)}\right)\,.$$
Note that the right side is a union of finitely many hyperplanes of dimension $r+s-1$ so an element $v_{s+k+1}^\prime$ does indeed exist. Set $\frak{B}_{k+1}=\{v_1^\prime,\cdots,v_{s+k+1}^\prime\}$\,.
\end{itemize}
It is now easy to verify that the set $\frak{B}_{s}$ is a strongly admissible basis.
\end{proof}

\section{Coleman-adapted Kolyvagin systems}
\label{sec:appendixKSexists}
Throughout this Appendix, let $F$ be a totally real or a CM field as above. Let $\frak{O}$ be the ring of integers of a finite extension $\Phi$ of $\QQ_p$, with maximal ideal $\frak{m}$, residue field $k$ and uniformizer $\varpi$. Let $T$ be a $G_F$-stable $\frak{O}$-lattice inside $\mathcal{M}_p(\eta^{-1})$, the twist of the the $p$-adic realization of a motive $\mathcal{M}$ (of the sort considered in the main body of this article) by an even Dirichlet character $\eta$ of $\Delta$. Then $T$ is a free $\frak{O}$-module of finite rank which is equipped with a continuous $G_F$-action unramified outside a finite set of places $\Sigma$ of $F$. Set $\overline{T}=T/\frak{m} T$. We assume that all places of $F$ at infinity and above $p$ are contained in $\Sigma$. We assume that $T$ verifies the hypotheses \textbf{(H1)}-\textbf{(H4)} of \cite[Section 3.5]{mr02} as well as the following:

\textbf{(H.Tam)} For every finite place $\lambda \in \Sigma$, the module $H^0(I_\lambda,T\otimes\Phi/\ooo)$ is divisible. Here $I_\lambda$ stands for the inertia group at the prime $\lambda$.

\textbf{(H.nE)} For every prime $\frak{p}\mid p$ of $F$, we have
$$H^0(F_\frak{p},T)=H^2(F_{\frak{p}},T)=0\,.$$

In this appendix we let $F_\infty$ denote the cyclotomic $\ZZ_p$ extension of $F$ and $\Gamma=\textup{Gal}(F_\infty/F)$. Note that this is the pro-$p$ part of the group considered in the main text. Let $\Lambda^{(p)}=\ooo[[\Gamma]]$. Let $\mathbb{T}=T\otimes\LLp$ and fix  $\uI \in \frak{I}_p$ as in the conclusion of Proposition~\ref{prop:specialelementstransversaltopmsubmodules}. To ease notation, we will set $R=\LLp$ and $d=g_-$. {We fix throughout an $\uI \in \fI_p$ verifying the conclusion of  Proposition~\ref{prop:specialelementstransversaltopmsubmodules} and associated to this choice, fix a signed Coleman map
\begin{equation}\label{eqn:def:normalizedcolemanmap}\frak{C}:=\col_{\mathcal{M}_p}^{\uI,\eta}: H^1(F_p,\mathbb{T}){\longrightarrow}R^d\,\,.\end{equation}
Here $\col_{\mathcal{M}_p}^{\uI}$ here corresponds to the the Coleman map denoted by $\col_{T(\eta)}^{\uI}$ in the main text and $\col_{\mathcal{M}_p}^{\uI,\eta}$ is its restriction to $\eta$-isotypic component.  Let $Z\subset R^d$ denote a $R$-submodule of the target of the Coleman map $\frak{C}$ such that 
\begin{itemize}
\item $Z$ is free of rank $d$.
\item The $R$-module $Z/\textup{im}(\frak{C})$ is pseudo-null.
\end{itemize}
The existence of such $Z$ is guaranteed by  Corollary~\ref{cor:pseudosurj2}.} 

We now fix an arbitrary rank-one direct summand  $\mathbb{L} \subset Z$. 
\begin{defn}
\label{def:LrestrictedSelmerstructure}
 Let $\FFF_{\mathbb{L}}$ denote the Selmer structure on $\mathbb{T}$ given with the following data: 
\begin{itemize}
\item $H^1_{\FFF_{\mathbb{L}}}(F_\lambda,\mathbb{T})=H^1(F_{\lambda}, \mathbb{T})$ for primes $\lambda\nmid p$, 
\item $H^1_{\FFF_{\mathbb{L}}}(F_p,\mathbb{T})=\ker\left(H^1(F_p,\mathbb{T})\stackrel{\frak{C}}{\longrightarrow}Z/\mathbb{L}\right)\,.$
\end{itemize}
Let $\mathcal{P}$ be the set of places of $F$ that does not contain the archimedean places, primes at which $T$ is ramified and primes above $p$. Finally let $\overline{\mathbf{KS}}(\mathbb{T},\FFF_{\mathbb{L}},\mathcal{P})$ be the $R$-module of generalized Kolyvagin systems defined as in \cite[Section 3.2.2]{kbbdeform}. An element of this module will be called an \emph{$\mathbb{L}$-restricted Kolyvagin system}.

We also let $\mathcal{F}_{\mathbb{L}}^*$ denote the dual Selmer structure on the Cartier dual $\mathbb{T}^\dagger$, in the sense of \cite[Definition 1.3.1 and \S2.3]{mr02}.
\end{defn}
As in the main body of thus text, we assume the truth of the weak Leopoldt conjecture for $T$. Our goal in this appendix is to give a proof of Theorem~\ref{thm:appendixBmain}. 

\begin{lemma}
\label{lem:auxiliarypropagation}
Suppose $R$ is any commutative ring and $M,N,Q$ are finitely generated $R$-modules such that we have an exact sequence 
$$0\longrightarrow M\stackrel{\iota}{\longrightarrow } N\longrightarrow Q$$
and the quotient $N/\iota(M)$ is $R$-torsion-free. For any ideal $I$ of $R$, let $X_I=X\otimes_R R/I$ for $X=M,N,R$. Then the following sequence of $R_I$-modules is exact:
$$0\longrightarrow M_I \stackrel{\iota_I}{\longrightarrow } N_I\longrightarrow Q_I\,\,.$$
\end{lemma}
\begin{proof}
Suppose $m \in M$ is such that $\iota(m) \in I\cdot N$, say $\iota(m)=r\cdot n_0$ for some $r\in I$ and $n_0 \in N$. As the quotient $N/\iota(M)$ is $R$-torsion-free, it follows that $n_0 \in \iota(M)$, say $n_0=\iota(m_0)$. Thus $\iota(m)=\iota(r\cdot m_0)$ and since $\iota$ is injective, $m\in I\cdot M$. We just proved that $I\cdot M=\ker\left(M\stackrel{\iota}{\longrightarrow } N_I\right)$
which is equivalent to the assertion of the Lemma.
\end{proof}

\begin{lemma}
\label{lem:theLrestrictedSelmergroupisfreeofcorrectrank}
The $R$-module $H^1_{\FFF_{\mathbb{L}}}(F_p,\mathbb{T})$ is free of rank $g_+$+1.
\begin{proof}
Let $L$ denote the image of $\mathbb{L}$ (resp., $\overline{Z}$ the image of $Z$) under the augmentation map $\frak{A}:R\twoheadrightarrow \ooo$. Observe the commutative diagram
$$\xymatrix{
0\ar[r]&H^1_{\FFF_{\mathbb{L}}}(F_p,\mathbb{T})\ar[r]\ar@{.>}[d] &H^1_{\Iw}(F_p,T)\ar[r]^(.6){\frak{C}}\ar@{>>}[d]^(.48){\otimes_\frak{A}\ooo}&Z/\mathbb{L}\ar@{>>}[d]^(.48){\otimes_\frak{A}\ooo}\\
0\ar[r]&\ker(\frak{C}_\frak{A})\ar[r] &H^1(F_p,{T})\ar[r]^(.6){\frak{C}_\frak{A}}&\overline{Z}/{L}
}$$
where $\frak{C}_\frak{A}:=\frak{C}\otimes_\frak{A}\ooo$ is the induced map on $H^1_{\Iw}(F_p,{T})\otimes_{\frak{A}}\ooo \stackrel{\sim}{\longrightarrow}H^1(F_p,{T})$. As the cokernel of $\frak{C}$ is finite so is the cokernel of $\frak{C}_\frak{A}$  
and it follows that $\ker(\frak{C}_\frak{A})$ is a free $\ooo$-module of rank $g_++1$ and by Nakayama's lemma that the $R$-module $H^1_{\FFF_{\mathbb{L}}}(F_p,\mathbb{T})$ is  generated by at most $g_++1$ elements. On the other hand, the first row of the diagram above shows that the generic fiber of $H^1_{\FFF_{\mathbb{L}}}(F_p,\mathbb{T})$ has rank $g_++1$ hence, together with our the discussion above,  we conclude that the $R$-module $H^1_{\FFF_{\mathbb{L}}}(F_p,\mathbb{T})$ is generated by exactly $g_++1$ elements. It is not hard to see (using the fact that $R$ is a UFD) that these generators cannot satisfy a non-trivial $R$-linear relation.
\end{proof}
\end{lemma}
\begin{theorem}
\label{thm:appendixBmain}
Let $\mathcal{P}_{1,\bar{1}} \subset \mathcal{P}$ be as in Definition~\ref{def:kolyprimes} below.
\begin{itemize}
\item[(i)] The $R$-module $\overline{\mathbf{KS}}(\mathbb{T},\FFF_{\mathbb{L}},\mathcal{P})$ is free of rank one, generated by \emph{any} Kolyvagin system $\pmb{\kappa}$ whose image $\overline{\pmb{\kappa}} \in {\mathbf{KS}}(\overline{T},\FFF_{\mathbb{L}},\mathcal{P}_{1,\bar{1}})$ is non-zero.
\item[(ii)] For an arbitrary generator $\{\kappa_n\}=\pmb{\kappa}$, the leading term $\kappa_1\in H^1_{\FFF_{\mathbb{L}}}(F,\mathbb{T})$ is non-vanishing.
\item[(iii)] Suppose $\{\kappa_n\}=\pmb{\kappa}\in \overline{\mathbf{KS}}(\mathbb{T},\FFF_{\mathbb{L}},\mathcal{P}_X)$ is a generator. Then, 
$$\textup{char}\left(H^1_{\FFF_{\mathbb{L}}}(F,\mathbb{T})/\Lambda\cdot\kappa_1\right)= \textup{char}\left(H^1_{\FFF_{\mathbb{L}}^*}(F,\mathbb{T}^\dagger)^{\vee}\right).$$
\end{itemize}
\end{theorem}
It is the statement of Theorem~\ref{thm:appendixBmain}(iii) that is key to all our results towards Perrin-Riou's main conjectures. 

Proof of the parts (i) and (iii) of Theorem~\ref{thm:appendixBmain} is identical to the proof of \cite[Theorem~A.12]{kbbCMabvar}\footnote{In fact, both proofs rely on the arguments of~\cite{kbbdeform} where a similar statement was proved in much more general context.} once we verify that the analogous statement to Definition/Theorem A.9 in loc.cit. holds true in our setting and that the core Selmer rank $\chi(\overline{T},\FFF_{\mathbb{L}})$ (in the sense of \cite[Definition 4.1.11]{mr02}) of the Selmer structure $\FFF_{\mathbb{L}}$ on $\overline{T}$ is 1. The first of these is achieved in Theorem~\ref{Adefthm:cartesian} below and the second in Proposition~\ref{prop:corerankisone}. The main difficulty is that the images of the Coleman maps are not necessarily free.

We first provide a proof of (ii) here.
\begin{proof}[Proof of Theorem~\ref{thm:appendixBmain}(ii)] Thanks to our choice of $\uI \in \fI_p$ and Proposition~\ref{prop:Ttorsion}, note that the modified Selmer group $\Sel_{\uI}(T^\dagger/F(\mu_{p^\infty}))^\Delta$ is $R$-cotorsion. Thus the $R$-module $H^1_{\FFF_{\mathbb{L}}^*}(F,\mathbb{T}^\dagger) \subset \Sel_{\uI}(T^\dagger/F(\mu_{p^\infty}))^\Delta$ is cotorsion as well. We may now conclude the proof using~\cite[Theorem 5.10]{kbbdeform}.
\end{proof}

Before settling Theorem~\ref{thm:appendixBmain} in full, we introduce the necessary terminology mostly borrowed from \cite{mr02}. Fix a topological generator $\gamma$ of the group $\Gamma$. 
We then have a (non-canonical) isomorphism $R\cong \frak{O}[[\gamma-1]]$.
\begin{defn}
\label{adefine:quotients}
For $k, \alpha \in \ZZ^+$, set 
$$R_{k,\alpha}:=R/(\varpi^k,(\gamma-1)^{\alpha}),$$
$$\mathbb{T}_{k,\alpha}:=\mathbb{T}\otimes_{R}R_{k,\alpha}=\mathbb{T}/(\varpi^k,(\gamma-1)^{\alpha})$$
and define the collection 
$$\textup{Quot}(\mathbb{T}):=\{\mathbb{T}_{k,\alpha}: k,\alpha \in \ZZ^+\}.$$
The propagation of the Selmer structure $\FFF_{\mathbb{L}}$ (in the sense of \cite[Example 1.1.2]{mr02}) to the quotients $\mathbb{T}_{k,\alpha}$ will still be denoted by the symbol $\FFF_{\mathbb{L}}$ as well as its propagation to $T$. 
\end{defn}
\begin{defn}
\label{def:kolyprimes}For $k, \alpha \in\ZZ^+$ define
\begin{itemize}
\item[(i)]  $H_{k,\alpha}=\ker\left(G_F\rightarrow \textup{Aut}(\mathbb{T}_{k,\alpha})\oplus\textup{Aut}(\pmb{\mu}_{p^{k}})\right)$,
\item[(ii)]  $L_{k,\alpha}=\overline{F}^{H_{k,\alpha}}$,
\item[(iii)] $\mathcal{P}_{k,\alpha}=\{\hbox{Primes }\lambda \in \mathcal{P}_{X} : \lambda \hbox{ splits completely in } L_{k,\alpha}/F\}$.
\end{itemize}
The collection $\mathcal{P}_{k,\alpha}$ is called the collection of \emph{Kolyvagin primes} for $\mathbb{T}_{k,\alpha}$. Define $\mathcal{N}_{k,\alpha}$ to be the set of square free products of primes in $\mathcal{P}_{k,\alpha}$.
\end{defn}

\begin{defn}
\label{adefine:transverseandmodifiedselmeratfrakn}$\,$\\
(i) Given $\lambda \in \mathcal{P}_{k,\alpha}$ fix once and for all an abelian extension $F^\prime/F_\lambda$ which is totally and tamely ramified, and moreover is a maximal such extension. As in~\cite[Definition 1.1.6(iv)]{mr02},  the \emph{transverse local condition} at $\lambda$ is defined to be
$$H^1_{\textup{tr}}(F_\lambda,T_{k,\alpha})=\ker\{H^1(F_\lambda,T_{k,\alpha}) \longrightarrow H^1(F^\prime,T_{k,\alpha})\}.$$
(ii) For $\frak{n} \in  \mathcal{N}_{k,\alpha}$, define the Selmer structure $\FFF_{\mathbb{L}}(\frak{n})$ on $T_{k,\alpha}$ by setting
$$H^1_{\FFF_{\mathbb{L}}(\frak{n})}(F_\lambda,T_{k,\alpha})=\left\{
\begin{array}{cr}
H^1_{\FFF_{\mathbb{L}}}(F_\lambda,T_{k,\alpha}), & \hbox{ if } \lambda\nmid \frak{n},\\\\
 H^1_{\textup{tr}}(F_\lambda,T_{k,\alpha}), &\hbox{ if } \lambda\mid \frak{n}.
\end{array}
\right.$$
\end{defn}
The following list of properties is key in proving Theorem~\ref{thm:appendixBmain}.

\begin{theorem}
\label{Adefthm:cartesian}
For any $\frak{n} \in \NN_{k,\alpha}$ the Selmer structure $\FFF_{\mathbb{L}}(\frak{n})$ is cartesian on the collection $\textup{Quot}(\mathbb{T})$ in the following sense. Let $\lambda$ be any prime of $F$.
\begin{enumerate}
\item[\textbf{(C1)}] \textup{(Functoriality)} For $\alpha\leq\beta$ and $k\leq k^\prime$,
$H^1_{\FFF_{\mathbb{L}}(\frak{n})}(F_{\lambda},\mathbb{T}_{k,\alpha})$ is the exact image of $H^1_{\FFF_{\mathbb{L}}(\frak{n})}(F_\lambda,\mathbb{T}_{k^{\prime},\beta})$ under the canonical map $H^1(F_{\lambda},\mathbb{T}_{k^\prime,\beta}) \rightarrow H^1(F_{\lambda},\mathbb{T}_{k,\alpha}).$

\item[\textbf{(C2)}] \textup{(Cartesian property along the cyclotomic tower)} 
$$H^1_{\FFF_{\mathbb{L}}(\frak{n})}(F_{\lambda},\mathbb{T}_{k,\alpha})=\ker\left( H^1(F_{\lambda},\mathbb{T}_{k,\alpha})\longrightarrow \frac{H^1(F_{\lambda},\mathbb{T}_{k,\alpha+1})}{H^1_{\FFF_{\mathbb{L}}(\frak{n})}(F_{\lambda},\mathbb{T}_{k,\alpha+1})}\right).$$
Here the arrow is induced from the injection $\mathbb{T}_{k,\alpha}\stackrel{[\gamma-1]}{\longrightarrow }\mathbb{T}_{k,\bar{\alpha+1}}$ and $[\gamma-1]$ is the multiplication by $\gamma-1$ map.
\item[\textbf{(C3)}] \textup{(Cartesian property as powers of $p$ vary)}  
$$H^1_{\FFF_{\mathbb{L}}(\frak{n})}(F_\lambda,\mathbb{T}_{k,\alpha})=\ker\left(H^1(F_\lambda,\mathbb{T}_{k,\alpha})\stackrel{[\varpi]}{\longrightarrow } \frac{H^1(F_\lambda,\mathbb{T}_{k+1,\alpha})}{H^1_{\FFF_{\mathbb{L}}(\frak{n})}(F_\lambda,\mathbb{T}_{k+1,\alpha})}\right),$$
    where the arrow is induced from the injection $\mathbb{T}_{k,\alpha} \stackrel{[\varpi]}{\longrightarrow } \mathbb{T}_{k+1,\alpha}.$
\end{enumerate}
\end{theorem}

\begin{proof}
For the primes $\lambda \nmid \frak{n}p$, the asserted properties may be verified as in \cite[\S2.3.1]{kbb}. The key points are the fact that the inertia group $I_\lambda \subset G_F$ acts trivially on $\Lambda^{(p)}$ and that we assumed \textbf{(H.Tam)}. For the primes $\lambda \mid \frak{n}$, they may proved as in \cite[\S4.1.4]{kbbdeform} (which itself, in this particular case of interest, is a slight generalization of \cite[Proposition 2.21]{kbb}). 

It therefore remains to verify the claimed properties at primes above $p$. The property \textbf{(C1)} is evident by definitions. Using Lemma~\ref{lem:auxiliarypropagation} one has a natural identification  for every $k,\alpha \in \ZZ^+$:  
$$H^1_{\FFF_{\mathbb{L}}(\frak{n})}(F_p,\mathbb{T}_{k,\alpha})=H^1_{\FFF_{\mathbb{L}}}(F_p,\mathbb{T})\otimes_{R}R_{k,\alpha}$$
(that is to say in more precise terms, $H^1_{\FFF_{\mathbb{L}}(\frak{n})}(F_p,\mathbb{T}_{k,\alpha})$ is the image of $H^1_{\FFF_{\mathbb{L}}}(F_p,\mathbb{T})$ under the obvious map). Note that Lemma~\ref{lem:auxiliarypropagation} applies with $M=H^1_{\FFF_{\mathbb{L}}}(F_p,\mathbb{T})$ and $N=H^1(F_p,\mathbb{T})$ as the quotient $H^1(F_p,\mathbb{T})/H^1_{\FFF_{\mathbb{L}}}(F_p,\mathbb{T})$ is $R$-torsion free by construction. The properties \textbf{(C2)} and \textbf{(C3)} follow now at once using the fact that the $R$-module $H^1_{\FFF_{\mathbb{L}}}(F_p,\mathbb{T})$ is free (of rank $g_++1$) by Lemma~\ref{lem:theLrestrictedSelmergroupisfreeofcorrectrank}.
\end{proof}

Let $\FFF_{\textup{null}}$ denote the Selmer structure on $\mathbb{T}$ given with the following data: 
\begin{itemize}
\item $H^1_{\FFF_{\textup{null}}}(F_\lambda,\mathbb{T})=H^1(F_{\lambda}, \mathbb{T})$ for primes $\lambda\nmid p$, 
\item $H^1_{\FFF_{\textup{null}}}(F_p,\mathbb{T})=H^1_{\uI}(F_p,T):=\ker\left(H^1(F_p,\mathbb{T})\stackrel{\frak{C}}{\longrightarrow}Z\right)\,.$
\end{itemize}
The assertion concerning the Selmer structure $\FFF_{\mathbb{L}}$ in the following Corollary follows immediately by Theorem~\ref{Adefthm:cartesian}. We need the statement on $\FFF_{\textup{null}}$ in our companion article \cite{buyukboduklei1} and it follows easily by modifying Lemma~\ref{lem:theLrestrictedSelmergroupisfreeofcorrectrank} appropriately.
\begin{corollary}
\label{for:appendixcartesianwhenpowersofpvary}
Propagations of both Selmer structures $\FFF_{\mathbb{L}}$ and $\FFF_{\textup{null}}$ on $T$ verify the hypothesis \textbf{\upshape H6} of \cite{mr02}.
\end{corollary}

\begin{proposition}
\label{prop:corerankisone}
The core Selmer rank $\chi(\overline{T},\FFF_\mathbb{L})$ equals $1$ whereas $\chi(\overline{T},\FFF_{\textup{null}})$ equals $0$. 
\end{proposition}
\begin{proof}
The proof of this proposition is similar to the proof of Proposition 9.2 in \cite{buyukboduklei1}. Let $\FFF_{\textup{can}}$ denote the canonical Selmer structure on $\mathbb{T}$ given with the data $H^1_{\FFF_{\textup{can}}}(F_\lambda,\mathbb{T})=H^1(F_{\lambda}, \mathbb{T})$ for every prime $\lambda \in \Sigma$\,. Using the global duality argument in \cite[Proposition 1.6]{wiles} and Corollary~\ref{for:appendixcartesianwhenpowersofpvary} we conclude that
$$\chi(\overline{T},\FFF_\textup{can})-\chi(\overline{T},\FFF)=\textup{rank}_R\, H^1_{\Iw}(F_p,{T})-\textup{rank}_R\,H^1_{\FFF}(F_p,\mathbb{T})$$ 
for $\FFF=\FFF_\mathbb{L}$ or ${\FFF_{\textup{null}}}$. However $\textup{rank}_R\, H^1_{\Iw}(F_p,T)=g$ and $\chi(\overline{T},\FFF_\textup{can})=g_-$ (c.f., \cite[Theorem 5.2.15]{mr02}) hence
$$\chi(\overline{T},\FFF)=\textup{rank}_R\,H^1_{\FFF}(F_p,\mathbb{T})-g_+\,.$$
The first part of the proposition follows by Lemma~\ref{lem:theLrestrictedSelmergroupisfreeofcorrectrank} and the second part using its appropriate generalization to apply with ${\FFF_{\textup{null}}}$.
\end{proof}
\subsection{The module of Kolyvagin determinants} Let us choose a basis 
$B=\{\phi_1,\cdots, \phi_{d-1}\}$ of the free $R$-module $\textup{Hom}_{R}\left(Z{/}\mathbb{L}\,,\,R\right).$ We then have an isomorphism
$\bigoplus_{i=1}^{d-1}\phi_i:R^d/\mathbb{L} \stackrel{\sim}{\rightarrow} R^{d-1}\,.$
Let $\widetilde{\phi}_i\in \textup{Hom}_{R}\left(Z,R\right)$ denote the pullback of $\phi_i$ with respect to the obvious projection. Note that the map $\phi:= \bigoplus_{i=1}^{d-1} \widetilde{\phi}_i \,\,:Z \rightarrow R^{d-1}$ is surjective with kernel $\mathbb{L}$. Define 
$$\Phi:=\widetilde{\phi}_1\wedge \cdots\wedge \widetilde{\phi}_{d}\in \wedge^{d}  \textup{Hom}_{R}\left(Z,R\right),$$
where the exterior product is taken in the category of $R$-modules. Let 
$$\Psi \in \wedge^{d}\,\textup{Hom}_{R}\left(H^1(F_p,\mathbb{T}),R\right)$$ 
be the pullback of $\Phi$ with respect to the Coleman map $\frak{C}$. 

\begin{proposition}
\label{prop:almostthere}
\begin{itemize}
\item[(i)] The map $\Phi$ maps $\wedge^d Z$ isomorphically onto $\mathbb{L}$.
\item[(ii)] For every $c \in \wedge^d H^1(F_p,\mathbb{T})$ we have $\Psi(c) \in H^1_{\mathcal{F}_\mathbb{L}}(F_p,\mathbb{T})$. 
\item[(iii)]Furthermore, the map $\Psi$ induces a map (which we still denote by $\Psi$)
$$\Psi:\,H^1(F_p,\mathbb{T})/H^1_{\uI}(F_p,T)\longrightarrow  H^1_{\mathcal{F}_\mathbb{L}}(F_p,\mathbb{T})/H^1_{\uI}(F_p,T)\,.$$
\end{itemize}
\end{proposition}
\begin{proof}
Linear Algebra.
\end{proof}
Proposition~\ref{prop:almostthere} may be summarized via the following commutative diagram:
\begin{equation}
\label{eqn:diagramappendixC}
\begin{gathered}
\xymatrix{\wedge^{d}Z\ar[r]^(.51){\Phi}_(.51){\sim}&\mathbb{L}\\
 \wedge^d{\displaystyle\frac{H^1(F_p,\mathbb{T})}{H^1_{\uI}(F_p,{T})}}\ar[r]^(.52){\Psi}\ar@{^{(}->}[u]_(.61){\frak{C}^{\otimes d}}&{\displaystyle\frac{H^1_{\mathcal{F}_\mathbb{L}}(F_p,\mathbb{T})}{H^1_{\uI}(F_p,{T})}}\ar@{^{(}->}[u]_(.6){\frak{C}}\\
 \wedge^{d}H^1(F,\mathbb{T})\ar@{^{(}->}[r]^(.52){\Psi}\ar@{^{(}->}[u]_(.43){\textup{loc}_p^{\otimes{d}}}&H^1_{\mathcal{F}_\mathbb{L}}(F,\mathbb{T})\ar@{^{(}->}[u]_(.43){\textup{loc}_p}
 }
 \end{gathered}
 \end{equation}
{The facts that $\Psi$ in the third row and $\textup{loc}_p^{\otimes{d}}$ are both surjective follow from the following proposition. 
 \begin{proposition}
\label{prop:bottomrowisfreeofrankone}
Under our running assumptions both $R$-modules $\wedge^d H^1(F,\mathbb{T})$ and $H^1_{\mathcal{F}_\mathbb{L}}(F,\mathbb{T})$ are free of rank one.
\end{proposition}}
\begin{proof}
It follows from the weak Leopoldt conjecture for $T$ (which we assume) that the $R$-module $H^1_{\mathcal{F}_{\textup{can}}^*}(F,\mathbb{T}^*)^\vee$ is torsion, where the canonical Selmer structure ${\mathcal{F}_{\textup{can}}}$ of Mazur and Rubin is given in the proof of Proposition~\ref{prop:corerankisone}. By control theorem (which holds true for this Selmer structure), we may find a  specialization $\pi: R\twoheadrightarrow \ooo$ (whose kernel is necessarily principal, say generated by $\varpi \in R$) such that $H^1_{\mathcal{F}_{\textup{can}}^*}(F,{T}_\pi^*)$, where $T_\pi:=\mathbb{T}\otimes_\pi \ooo$. By \cite[Theorem 5.2.15]{mr02}, it follows that $H^1_{\mathcal{F}_{\textup{can}}}(F,\mathbb{T})$ is an $\ooo$-module of rank $g$, which is also torsion-free (hence free) by our running assumptions.

Consider the natural injection $H^1(F,\mathbb{T})/\varpi  H^1(F,\mathbb{T}) \hookrightarrow  H^1_{\mathcal{F}_{\textup{can}}}(F,T_\pi)$. Using Nakayama's lemma, we see that $H^1(F,\mathbb{T})$ may be generated by the lifts of a basis of $H^1_{\mathcal{F}_{\textup{can}}}(F,T_\pi)$. Relying on the fact that $R$ is a UFD, one may further verify that these generators may not satisfy a non-trivial $R$-linear relation. This completes the proof of the assertion that $\wedge^d H^1(F,\mathbb{T})$ is free of rank $1$. The rest is proved in an identical manner.

\end{proof}
\begin{defn}
\label{define:theKolyvaginconstructeddeterminantelement}
\begin{itemize}
\item[(i)] Define the $\Lambda$-module of \emph{Kolyvagin leading terms} $\frak{L}(T)$ by setting
$$\frak{L}(T)=\left\{\sum_{\chi \in \widehat{\Delta}^+}{\kappa_1^{\chi}\cdot e_\chi} \in H^1_{\textup{Iw},S}(F,T):\,\, \{\kappa_{\frak{n}}^{\chi}\}=\pmb{\kappa}^\chi \in \overline{\mathbf{KS}}(\mathbb{T}(\chi),\FFF_{\mathbb{L}},\mathcal{P})\right\}\,.$$
Here $\widehat{\Delta}^+$ denotes the set of even characters of $\Delta$ and $e_\chi \in \ZZ_p[\Delta]$ the idempotent corresponding to $\chi$. It is not hard to see using Theorem~\ref{thm:appendixBmain} (for each twist $T(\chi)$) that the $\Lambda$-module $\frak{L}(T)$ is free of rank 1.
\item[(ii)] The $\Lambda$-module of \emph{Kolyvagin determinants} $\frak{K}(T)$ is defined as
$$\frak{K}(T)=\left\{\Xi \in \wedge^d  H^1_{\textup{Iw},S}(F,T):\,\Psi(\Xi) \in \frak{L}(T)\right\}.$$ 
\end{itemize}
\end{defn}
\begin{remark}
The diagram~(\ref{eqn:diagramappendixC}) above and the fact that $\frak{C}$ has pseudo-null cokernel show that $\frak{K}(T)\neq0$. One may also prove that this module does not depend on any of the choices made above and depends only on $T$. A suitable extension of the theory of higher rank Kolyvagin systems (as studied in \cite{mrksrankr}) over coefficient rings of dimension larger than 1 would yield a more natural definition of $\frak{K}(T)$. We plan to get back to this point in the future. 
\end{remark}

\section{Comparison with works of Kobayashi and Pollack}
\label{appendix:kobandrob}
{We shall compare the signed Selmer groups that we denoted by $\textup{Sel}_{\uI}$ in the main body of the article to the $\pm$-Selmer groups of Kobayashi \cite{kobayashi03}; and the $\uI$-signed $p$-adic $L$-functions to $\pm$-$p$-adic $L$-functions of Pollack \cite{pollack03}. In particular, we shall justify that our theory offers a natural generalization of their work.}

{Throughout this appendix, we assume that the motive $\mathcal{M}=h^1(E)(1)$ is associated to an elliptic curve $E/\QQ$ that has good supersingular reduction at $p$ and that $a_p(E)=0$, so that the $p$-adic realization $T$ of $\mathcal{M}$ will be the  $p$-adic Tate module of $E$ and the Pontryagin dual $T^\dagger$ is the $p$-divisible group $E[p^\infty]$. Note that in this case $g_-=1$ and we no longer fix an admissible basis. As it shall be clear from the discussion below, Lemma~\ref{lem:Timage} follows already from the work of Kobayashi and the second named author even if the basis of the Dieudonn\'e module is no  longer strongly admissible.}

\subsection{Kobayashi's $\pm$-Selmer groups}
{ Kobayashi in \cite{kobayashi03} defined the $\pm$-Selmer groups $\Sel_p^\pm(E/\QQ(\mu_{p^\infty}))$ by properly modifying the Bloch-Kato conditions at $p$. This is exactly what we do in Definition~\ref{def:modifiedselmergroup}, except that we used as our local conditions  at $p$ the submodules $H^1_{\uI}(\Qp(\mu_{p^\infty}),T^\dagger)$ in place of Kobayashi's submodules $E^\pm(\Qp(\mu_{p^\infty}))\subset E(\Qp(\mu_{p^\infty}))$ given by some ``jumping" trace conditions. Furthermore, as proved in \cite[\S4]{lei09}, Kobayashi's submodules may be realized as the orthogonal complements of the kernel of some $\pm$-Coleman maps $\col^\pm:\HIw(\Qp,T)\rightarrow\Lambda$, in the same way that the local conditions $H^1_{\uI}(\Qp(\mu_{p^\infty}),T^\dagger)$ in Definition~\ref{def:modifiedselmergroup} are defined as the orthogonal complement of $\ker(\col_{\uI})$. Therefore, in order to compare our $\Sel_{\uI}$ with Kobayashi's $\Sel_p^\pm$, it is enough to compare our Coleman maps $\col_{\uI}$ with the $\pm$-Coleman maps defined in \cite{kobayashi03}. Note that these were already rewritten in the language of Dieudonn\'e modules in \cite{lei09}.}
 
 {
 Let $\Dcris(T)=\DD_{\Qp}(T)$. We fix a basis $v_1\in\Fil^0\Dcris(T)$ and we extend it to a basis $v_1,v_2=\vp(v_1)$ of $\Dcris(T)$. The matrix of $\vp$ with respect to this basis is given by
 \[
 C_\vp=\begin{pmatrix}
 0& -1/p\\1&0
 \end{pmatrix}=\begin{pmatrix}
 0&-1\\1&0
 \end{pmatrix}\begin{pmatrix}
 1&0\\0&1/p
 \end{pmatrix}.
 \]
 Therefore, under the notation of Proposition~\ref{prop:matrix}, we find that the logarithmic matrix $M_T$ with respect to the same basis is given by
 \[
 M_T=\begin{pmatrix}
 0&-\log^+\\
 \log^-&0
 \end{pmatrix},
 \]
 where $\log^\pm$ are Pollack's $\pm$-logarithms defined by the formulae
 \begin{align*}
 \log^+&=\frac{1}{p}\prod_{n\ge1}\frac{\Phi_{p^{2n}}(1+X)}{p},\\
 \log^-&=\frac{1}{p}\prod_{n\ge1}\frac{\Phi_{p^{2n-1}}(1+X)}{p}.
 \end{align*}  Let $\col_1,\col_2$ be the two Coleman maps corresponding to this matrix as in Theorem~\ref{thm:decomposereg}. We have the relation
 \[
\calL_{T,1}=-\log^+\col_2\quad\text{and}\quad \calL_{T,2}=\log^-\col_1. 
 \]
 On combining this with \eqref{eq:regulator}, we may compare our Coleman maps with the $\pm$-Coleman maps defined in \cite[\S3.4]{lei09} and see that they differ simply by a minus sign, namely
 \begin{equation}
  \col^+=-\col_2\quad\text{and}\quad \col^-=\col_1.
\label{eq:agree}
 \end{equation}
 In particular they have the same kernels.
  \begin{remark}
 Note that this choice of basis of $\Dcris(T)$ is not admissible in the sense of Definition~\ref{def:admissiblebasis}. As noted in Remark~\ref{rem:whyadmissible}, this means that the images of our Coleman maps would not be pseudo-isomorphic to $\Zp[[X]]$. Indeed, as shown in \cite[Propositions~8.23 and 8.24]{kobayashi03}, $\col^+$ is surjective, while the isotypic component of $\image(\col^-)$ at a non-trivial character is $X\Zp[[X]]$. This is consistent with our Propositions~\ref{prop:image} and~\ref{prop:iageofcolemanis}.
 \end{remark}
  \subsection{Pollack's $\pm$-$p$-adic $L$-functions}
  In \cite[\S3.4]{lei09} as well as \cite[Theorem~6.3]{kobayashi03}, it has been showed that the Pollack's $\pm$-$p$-adic $L$ functions in \cite{pollack03} is the image of the Beilinson-Kato elements along the cyclotomic tower (as constructed in \cite{kato04}) under  the $\pm$-Coleman maps, up to a sign. Note in particular that the tower of Beilinson-Kato elements does satisfy Conjecture~\ref{conj:Tspecialelement}. Furthermore, the $\uI$-signed $p$-adic $L$-functions given as in Definition~\ref{def:signedpadicLfunc} are simply the image of the Beilinson-Kato elements under $\col_1$ and $\col_2$. Therefore, thanks to \eqref{eq:agree}, they agree with Pollack's $\pm$-$p$-adic $L$ functions, up to a sign.
}

\bibliographystyle{amsalpha}
\bibliography{references}
\end{document}